\documentclass[]{fundam}

\usepackage{enumerate,amsmath,amscd,amssymb,algorithm}
\usepackage[noend]{algpseudocode}
\usepackage{subfigure}
\usepackage{multirow}
\usepackage{colortbl}

\newcommand{\Z}{\mathbb{Z}}
\newcommand{\N}{\mathbb{N}}

\newcommand{\R}{\mathbb{R}}
\newcommand{\B}{\mathcal{B}}
\newcommand{\SIGMA}{\mathrm{\Sigma}}
\newcommand{\PI}{\mathrm{\Pi}}
\newcommand{\INF}{{}^\infty}
\newcommand{\A}{\mathcal{A}}
\newcommand{\F}{\mathcal{F}}
\newcommand{\PS}{PS}
\newcommand{\SP}{SP}
\newcommand{\rows}{PS}
\newcommand{\rank}{\mathrm{rank}}
\newcommand{\leftz}{\stackrel{0}{\leftarrow}}
\newcommand{\leftp}{\stackrel{1}{\leftarrow}}

\usepackage{tikz}

\usepackage{calc}

\newtheorem{conjecture}{Conjecture}

\algnewcommand{\Choose}[1]{\textbf{choose}~#1}
\algnewcommand{\Nop}[1]{\textbf{sleep}(#1)}

\newtheorem{construction}{Construction}

\newcommand{\picnote}[2]{
	\node[draw,fill=white] () at #1 {{\tiny #2}};
}

\newcommand{\diamondlord}[4]{
	\fill[#4,opacity=.15] (#1+#3,#2) -- (#1,#2+#3) -- (#1-#3,#2) -- (#1,#2-#3);
	\draw[#4,thick] (#1+#3,#2) -- (#1,#2+#3) -- (#1-#3,#2) -- (#1,#2-#3) -- cycle;
}

\newcommand{\green}[1]{\emph{#1}}
\newcommand{\red}[1]{#1}

\usepackage{xspace}
\newcommand{\CapPProperty}{Bounded signal property\xspace}
\newcommand{\PProperty}{bounded signal property\xspace}
\newcommand{\SoficPProperty}{bounded sofic signal property\xspace}

\begin{document}

\title{Constructions with countable subshifts of finite type\thanks{Research supported by the Academy of Finland Grant 131558}}


\author{Ville Salo\\
TUCS -- Turku Centre for Computer Science, Finland, \\
University of Turku, Finland, \\
vosalo{@}utu.fi
\and Ilkka T\"orm\"a\\
TUCS -- Turku Centre for Computer Science, Finland, \\
University of Turku, Finland, \\
iatorm{@}utu.fi}

\maketitle

\runninghead{V. Salo, I. T\"orm\"a}{Constructions with countable subshifts of finite type}

\begin{abstract}
We present constructions of countable two-dimensional subshifts of finite type (SFTs) with interesting properties. Our main focus is on properties of the topological derivatives and subpattern posets of these objects. We present a countable SFT whose iterated derivatives are maximally complex from the computational point of view, constructions of countable SFTs with high Cantor-Bendixson ranks, a countable SFT whose subpattern poset contains an infinite descending chain and a countable SFT whose subpattern poset contains all finite posets. When possible, we make these constructions deterministic, and ensure the sets of rows are very simple as one-dimensional subshifts.
\end{abstract}


\section{Introduction}

A large portion of the theory of multidimensional subshifts of finite type (SFTs), also commonly known as tiling systems, concentrates on aperiodic tile sets and additional properties such sets can have, as it is quite uncommon that a tiling system cannot have a given property, unless there is a trivial obstruction. Simply the problem of finding an aperiodic tile set has a variety of clever solutions in the literature \cite{Be66,Ro71,Ka96a,DuRoSh12}, and examples of other interesting constructions with tilings can be found in for example \cite{PaSc10,Ho09}. In order to find limitations to what SFTs can do, one approach is to restrict to a subclass of SFTs with an additional limitation, and prove that some further properties must follow. The results of \cite{BaDuJe08} can be considered to follow this schema: If we restrict ourselves to the class of nonempty countable SFTs, then we can always find doubly periodic configurations.

Of course, once such a restriction is introduced, it makes sense to find out more about the structure of the resulting subshifts, and to this end, \cite{BaDuJe08} leaves open two interesting questions: whether a Cantor-Bendixson rank more than $\omega$ can be obtained for a countable SFT, and whether the subpattern poset of a countable SFT can contain `only finitely many levels'. As $\omega$ (an infinite upward chain) cannot be embedded in the subpattern poset of a countable SFT, we interpret the second question as asking whether the reversed ordinal $-\omega$ (an infinite downward chain) can be embedded.

The problem of finding an infinite Cantor-Bendixson rank was solved in \cite{JeVa11} by drawing runs of a Turing machine on the configurations of the SFTs. Once it is shown that the SFT is countable, the existence of an infinite Cantor-Bendixson rank follows easily, as the Cantor-Bendixson ranks obtained are cofinal in the Cantor-Bendixson ranks of countable $\PI^0_1$ sets. This also implies that the construction is in a sense optimal. In \cite{SaTo12b}, we gave an alternative example of rank more than $\omega$ with a more geometric construction, and also showed a perhaps simpler alternative for embedding computations, using a counter machine instead of a Turing machine.

The problem of embedding an infinite downward chain (the total order $-\omega$) in the subpattern poset is a bit trickier. We constructed a sofic shift with this property in \cite{SaTo12b}, and conjectured that this is impossible to do with an SFT, mainly because our sofic example relied heavily on the uncountability of the SFT cover, and perhaps also because we hoped for this to be the case.

This paper is an extension of \cite{SaTo12b}. As our main result, we refute our conjecture from \cite{SaTo12b} by exhibiting a countable two-dimensional SFT whose subpattern poset contains an infinite downward chain. This means that both questions of \cite{BaDuJe08} have positive solutions, implying that the restriction of countability for SFTs is, in the end, rather weak. We also draw ideas from our earlier paper \cite{SaTo12c}, where we studied cellular automata on countable one-dimensional SFTs and sofic shifts. As a cellular automaton is little more than a deterministic tiling rule, \cite{SaTo12c} can be thought of as studying subshifts which are deterministic downward (or upward, depending on the convention) and whose horizontal projective subdynamics (sets of rows) are contained in a countable SFT. We refer to the latter property as the \PProperty. It turns out that most of the constructions in \cite{SaTo12b} already have the \PProperty, and we try to determinize our constructions whenever possible. Interestingly, the construction of an SFT with high Cantor-Bendixson rank is possible with all the three properties determinism, countability and the \PProperty, but we show that determinism \emph{or} countability, in conjunction with the \PProperty, already forbids an infinite descending chain.

In addition to the construction of an infinite descending chain and SFTs with high Cantor-Bendixson ranks, we include our construction of a countable SFT whose derivatives grow in complexity from \cite{SaTo12b}, and show how to make this deterministic. As further applications of the counter machine construction, we optimize some of the results of \cite{JeVa11}. On the side of posets, the construction in \cite{SaTo12b} of countable SFTs whose posets include embeddings of arbitrary finite posets is made deterministic and augmented with the \PProperty as well. Further, using these SFTs as building blocks, we build a single countable deterministic SFT whose subpattern poset contains all finite posets.

\section{Definitions and notation}

Let $S$ be a finite set of \emph{symbols}, called the \emph{alphabet}, endowed with the discrete topology. For an integer dimension $d \geq 1$, the set $S^{\Z^d}$, equipped with the product topology, is called the \emph{$d$-dimensional full shift on $S$}. Elements $x$ of $S^{\Z^d}$ are called \emph{configurations}. A configuration $x \in S^{\Z^d}$ is \emph{uniform} if there exists $s \in S$ with $x_{\vec n} = s$ for all $\vec n \in \Z^d$. A \emph{pattern over $S$} is a pair $(D,s)$, where $D \subset \Z^d$ is a finite \emph{domain}, and $s : D \to S$ gives the arrangement of symbols in $D$. A pattern $P = (D,s)$ \emph{occurs} in a configuration $x$, denoted $P \sqsubset x$, if we have $x_{D + \vec{n}} = s$ for some $\vec{n} \in \Z^d$. For all $\vec n \in \Z^d$, we define the \emph{translation} (known as the \emph{shift action}, if $d = 1$) $\sigma^{\vec n} : S^{\Z^d} \to S^{\Z^d}$ by $\sigma^{\vec n}(x)_{\vec m} = x_{\vec n + \vec m}$.

A \emph{$d$-dimensional subshift over $S$} is a closed subset $X \subset S^{\Z^d}$ satisfying $\sigma^{\vec n}(X) = X$ for all $\vec n \in \Z^d$. Alternatively, all subshifts $X$ can be defined by a set $\F$ of \emph{forbidden patterns} as $X = \{ x \in S^{\Z^d} \;|\; \forall P \in \F: P \not\sqsubset x \}$. If $\F$ is finite, then $X$ is said to be \emph{of finite type} (SFT for short). Given a finite domain $D \subset \Z^d$, the set of patterns occurring in the points of a subshift $X$ with domain $D$ is denoted $\B_D(X)$, the set of all patterns of $X$ is $\bigcup_D \B_D(X) = \B(X)$, and the set of symbols occurring in $X$ is denoted $\A(X)$. A \emph{block map} is a continuous mapping $f : X \to Y$, where $X$ and $Y$ are $d$-dimensional subshifts (possibly over different alphabets), which intertwines the shift maps of $X$ and $Y$: $f \circ \sigma^{\vec n} = \sigma^{\vec n} \circ f$ for all $\vec n \in \Z^d$. Alternatively, a block map $f$ can be defined by a \emph{local function} $F : \B_D(X) \to \A(Y)$ by $f(x)_{\vec n} = F(x)_{D + \vec n}$ for all $x \in X$ and $\vec n \in \Z^d$, where $D$ is a finite domain, called the \emph{neighborhood} of $f$ \cite{He69}. An image of a subshift under a block map is a subshift, and images of SFT's are called \emph{sofic shifts}. Two subshifts are called \emph{conjugate} if there is a bijective block map between them. See \cite[Section 13.10]{LiMa95} for a short survey on multidimensional symbolic dynamics.

A subshift is two-dimensional if not specified otherwise. In all informal descriptions, the y-axis is vertical and increases upward, and the x-axis increases to the right. That is, the words up and north refer to the vector $(0, 1)$, while right and east refer to $(1, 0)$. Rows are horizontal and columns are vertical, and for an integer $i \in \N$ and configuration $x \in S^{\Z^2}$, we denote by $x_i$ the $i$th row of $x$, that is, the configuration $y \in S^\Z$ defined by $y_j = x_{(j,i)}$.

For a subshift $X \subset S^\Z$, the block maps from $X$ to itself are called \emph{cellular automata}. The set of \emph{limit spacetime diagrams} of a cellular automaton $f : X \to X$ is the SFT
\[ \{ x \in S^{\Z^2} \;|\; \forall j \in \Z: x_j \in X \wedge x_j = f(x_{j+1}) \}. \]
The reason we do not call such configurations simply the spacetime diagrams of $f$ is that if $f$ is not surjective, only rows with an infinite chain of preimages appear in them, and the set of such rows is usually called the limit set of $f$.

To define a one-dimensional SFT or sofic shift, instead of supplying the forbidden patterns, we usually use the notation
\[ \B^{-1}(L) = \{ x \in S^\Z \;|\; \forall r: \exists u \in L: x_{[-r, r]} \sqsubset u \} \]
where $L$ is a regular language. As the notation $x_i$ for $x \in S^{\Z^2}, i \in \N$ implies, we think of a two-dimensional configuration as being a one-dimensional configuration with legal rows as the alphabet. In particular, we may use the notation $x_{[i,j]}$ to extract a finite list of rows from $x$.

We call a vector $(x,y) \in \Z^2$ with $\gcd(x,y) = 1$ a \emph{direction}, and we denote by $SL_d(\Z)$ the restriction of $SL_d(\R)$ to those functions that map $\Z^d$ bijectively to itself. In the case $d = 2$, they are given by $2 \times 2$ integer matrices of determinant $1$, and it follows from Bezout's identity that for any two directions $\vec d, \vec e \in \Z^2$, there exists an element $A \in SL_2(\Z)$ with $A(\vec d) = \vec e$. For a configuration $x \in S^{\Z^d}$ and $A \in SL_d(\Z)$, we define $A(x) \in S^{\Z^d}$ by $A(x)_{\vec n} = x_{A^{-1}(\vec n)}$ for all $\vec n \in \Z^d$, and for a subshift $X \subset S^{\Z^d}$, we define $A(X) = \{ A(x) \;|\; x \in X \}$. From the linearity of $A$ it follows that $A(X)$ is also a subshift. We fix a special element $A_{\frac{\pi}{2}} = \left( \begin{smallmatrix} 0 & -1 \\ 1 & 0 \end{smallmatrix} \right) \in SL_2(\Z)$ which performs a counterclockwise rotation by $\frac{\pi}{2}$.

We say that $X \subset S^{\Z^2}$ \emph{has a period} if for some vector $\vec n \in \Z^2 \setminus \{(0, 0)\}$ and for all $x \in X$, $x = \sigma^{\vec n}(x)$ holds, and $\vec n$ is then called the \emph{period} of $x$. If $X$ has period $\vec n$, and every period it has is a rational multiple of $\vec n$, then we say $X$ is \emph{singly periodic}. If $X$ has two periods which are linearly independent over $\R$, then we say $X$ is \emph{periodic}. We define periods of configurations, and periodicity and single periodicity of configurations analogously. For $y \in S^\Z$, we say $y$ is periodic if $y = \sigma^p(y)$ for some $p \in \N$, and we say $y$ is \emph{eventually periodic} if for some $n_0, p \in \N$, we have $y_i = y_{i + p}$ for all $i > n_0$ and $y_i = y_{i - p}$ for all $i < -n_0$. A two-dimensional configuration is horizontally (vertically) eventually periodic if its rows (columns) are eventually periodic.

Let $X \subset S^{\Z^2}$ be a subshift, and denote $H = \{ (a,b) \;|\; a \in \Z, b \geq 1 \}$. We say that $X$ is \emph{southward deterministic}, if whenever $x \in S^{\Z^2}$ is such that no forbidden pattern of $X$ occurs in $x|_H$, there exists at most one $y \in X$ with $x|_H = y|_H$. We say $X$ is \emph{extendably southward deterministic}, if there exists exactly one such $y$. By the translation invariance of $X$ and a compactness argument, south determinism is equivalent to the existence of an integer $i \geq 1$ such that $x_{\vec n} = y_{\vec n}$ for all $\vec n \in [-i, i] \times [1,i]$ implies $x_{\vec 0} = y_{\vec 0}$, when $x, y \in X$. We say $X$ is \emph{(extendably) deterministic in the direction $\vec d$} if $A(X)$ is (extendably) southward deterministic, where $A \in SL_2(\Z)$ is such that $A( A_{\frac{\pi}{2}} (\vec d)) = (1,0)$. Note that for Wang tiles, it is common to define northwest determinism as the property that a tile is uniquely defined by its north and west neighbors, whereas we would call the corresponding SFT southeast deterministic, since the tiles to the southeast of a known half plane are determined.

We define the \emph{projective subdynamics} of a subshift $X \subset S^{\Z^2}$ as the set of horizontal rows appearing in its configurations: $\PS(X) = \{ x_0 \;|\; x \in X \}$. For a general direction $\vec d \in \Z^2$, the \emph{$\vec d$-projective subdynamics} $\PS_{\vec d}(X)$ of $X$ is defined as $\PS(A(X))$, where $A \in SL_2(\Z)$ is such that $A( A_{\frac{\pi}{2}} (\vec d)) = (0,1)$. This means that we examine the contents of discrete lines perpendicular to $\vec d$, and in particular, $\PS(X) = \PS_{(0,1)}(X)$. For a configuration $x \in S^{\Z^2}$, we define $\rows(x) = \{ x_i \;|\; i \in \Z \}$, and generalize this to $\rows_{\vec d}(x)$ analogously. We say $X$ has the \emph{\PProperty} in direction $\vec d$, if for some $\tilde{X}$ conjugate to $X$ there exists a countable one-dimensional SFT $Y$ such that $\PS_{\vec d}(\tilde{X}) \subset Y$ (see Lemma~\ref{lem:UniqueParsing}). When the direction is not given, the default value $\vec d = (0, 1)$ is assumed. Intuitively, having the \PProperty in some direction means that the subshift can send information in that direction only using a bounded number of signals. The \PProperty in a fixed direction is, by definition, invariant under conjugacy. However, it is certainly possible that the rows of $X$ are not contained in a countable SFT but those of $\tilde{X}$ are, while $X$ and $\tilde{X}$ are conjugate. In our constructions, when we state that an SFT we construct has the \PProperty, we always mean that its set of rows is actually contained in a countable SFT. The definitions of determinism and the \PProperty via orthogonal directions are similar to the notion of \emph{slicing} as defined in \cite{DeFoWe13}, although we consider SFTs rather than cellular automata.

By a \emph{preorder} we mean a reflexive, transitive relation on a set $P$. If the preorder is antisymmetric, then $P$ with this order is called a \emph{poset}, or a partially ordered set. From any preorder $\leq$ on a set $P$, we obtain a poset $(\tilde{P}, \leq)$ where $\tilde{P}$ is the set of $\leq$-equivalence classes of $P$ (that is, the equivalence classes of the equivalence relation $(p \leq q) \wedge (q \leq p)$) and $\leq$ is the natural induced order among these classes. We often use the elements of $P$ and their $\leq$-equivalence classes interchangeably.

We obtain a natural preorder on $S^{\Z^d}$ by defining $x \leq y$ by $\B(x) \subset \B(y)$. To a subshift $X \subset S^{\Z^2}$, we associate its \emph{subpattern poset} $SP(X) = (\tilde{X}, \leq)$. Comparisons between points of $X$ refer to comparisons of the associated elements of the subpattern poset. In particular, chains in $X$ refer to chains in $\tilde{X}$.

We say a poset $P$ has the \emph{ascending chain condition}, or ACC, if $p_1 \leq p_2 \leq \cdots$ implies that for some $n \in \N$, $p_i = p_j$ for all $i, j \geq n$. That is, every ascending chain is eventually constant. Symmetrically, we define the \emph{descending chain condition} DCC. For posets $(P, \leq), (Q, \preccurlyeq)$, we denote by $P \boxtimes Q$ the poset $(P \dot{\cup} Q, (\leq) \dot{\cup} (\preccurlyeq))$, that is, the disjoint union of $P$ and $Q$ where the order among elements of $P$ is $\leq$, the order among elements of $Q$ is $\preccurlyeq$, and there are no relations between elements of $P$ and $Q$. For posets $((P_i, \leq_i))_i$, we similarly write $\boxtimes_i P_i$ to denote the disjoint union of all posets $P_i$. For a poset $(P, \leq)$, we write $-P$ for the poset $(P, \geq)$, where $p \geq q$ is defined by $q \leq p$. Note that a poset $P$ has the ACC if and only if $-P$ has the DCC. An \emph{order-embedding} of a poset $(P, \leq)$ to a poset $(Q, \preccurlyeq)$ is an injective function $\phi : P \to Q$ such that $p \leq q$ if and only if $\phi(p) \preccurlyeq \phi(q)$ for all $p, q \in P$. If $\phi$ is bijective, it is called an \emph{order-isomorphism}.

Let $X$ be a topological space. For every ordinal $\lambda$, we define the \emph{Cantor-Bendixson derivative of order $\lambda$} of $X$, denoted by $X^{(\lambda)}$, by transfinite induction:
\begin{itemize}
\item $X^{(0)} = X$,
\item $X^{(\alpha+1)} = \{x \in X^{(\alpha)} \;|\; x \mbox{ is not isolated in } X^{(\alpha)}\}$, and
\item $X^{(\alpha)} = \bigcap_{\beta < \alpha} X^{(\beta)}$, if $\alpha$ is a limit ordinal.
\end{itemize}
There must exist an ordinal $\lambda$ such that $X^{(\lambda)} = X^{(\lambda + 1)}$, as $X$ is a set. The lowest such $\lambda$ is called the \emph{Cantor-Bendixson rank} of $X$, and is denoted $\rank(X)$. From the definition of the derivative operator, it is clear that then $X^{(\rank(X))}$ is a perfect space. We note here that in a subshift $X \subset S^{\Z^d}$, a configuration $x$ is isolated if and only if there exists a pattern $(D,s) \in \B(X)$ such that $x$ is the only element of $X$ with $x_D = s$. We use this fact without any explicit mention in many of our proofs. We say that a topological space $X$ is \emph{ranked} if and only if $X^{(\rank(X))} = \emptyset$. The \emph{rank $\rank_X(x)$ of a point $x$} in a ranked topological space $X$ is the smallest ordinal $\lambda$ such that $x \notin X^{(\lambda)}$.

Let $\phi$ be a formula in first-order arithmetic. If $\phi$ contains only bounded quantifiers, then we say $\phi$ is $\SIGMA^0_0$ and $\PI^0_0$. For all $n > 0$, we say $\phi$ is $\SIGMA^0_n$ if it is equivalent to a formula of the form $\exists k: \psi$ where $\psi$ is $\PI^0_{n-1}$, and $\phi$ is $\PI^0_n$, if it is equivalent to a formula of the form $\forall k: \psi$ where $\psi$ is $\SIGMA^0_{n-1}$. This classification is called the \emph{arithmetical hierarchy} (see e.g. \cite[Chapter IV.1]{Od89} for an introduction to the topic). A subset $X$ of $\N$ is $\SIGMA^0_n$ or $\PI^0_n$, if $X = \{ x \in \N \;|\; \phi(x) \}$ for some $\phi$ with the corresponding classification. It is known that the $\SIGMA^0_1$ sets are exactly the recursively enumerable sets, and the $\PI^0_1$ sets their complements. When classifying sets of objects other than natural numbers (e.g. patterns), we assume that the objects are in some natural and computable bijection with $\N$. Also, a subshift is given the same classification as its language, so that, for example, two-dimensional SFTs are $\PI^0_1$ subshifts. See \cite{CeRe98} for a general survey on $\PI^0_1$ sets. The nonstandard quantifier $\exists^\infty n : \phi(n)$ has the meaning `there exist infinitely many $n$ such that $\phi(n)$.'

A subset $X \subset \N$ is \emph{many-one reducible} (or simply \emph{reducible}) to another set $Y \subset \N$, if there exists a computable function $f : \N \to \N$ such that $x \in X$ iff $f(x) \in Y$. If every set in a class $\mathcal{C}$ is reducible to $X$, then $X$ is said to be \emph{$\mathcal{C}$-hard}. If, in addition, $X$ is in $\mathcal{C}$, then $X$ is \emph{$\mathcal{C}$-complete}.

Let $k \in \N$. A \emph{$k$-counter machine} is defined as a quintuple $M = (k,\Sigma,\delta,q_0,q_f)$, where $\Sigma$ is a finite \emph{state set}, $q_0,q_f \in \Sigma$ the \emph{initial and final states} and
\[ \delta \subset (\Sigma \times [1,k] \times \{Z,P\} \times \Sigma) \cup (\Sigma \times [1,k] \times \{-1,0,1\} \times \Sigma) \]
the \emph{transition relation}. A \emph{configuration of $M$} is an element of $\Sigma \times \N^k$, with the interpretation of $(q,n_1,\ldots,n_k)$ being that the machine is in state $q$ with counter values $n_1,\ldots,n_k$. The machine operates in possibly nondeterministic steps as directed by $\delta$. If we have $(p, i, Z, q) \in \delta$ ($(p, i, P, q) \in \delta$), then from any configuration $(p,n_1,\ldots,n_k)$ such that $n_i = 0$ ($n_i > 0$, respectively), the machine $M$ may move to the configuration $(q,n_1,\ldots,n_k)$. If we have $(p, i, r, q) \in \delta$ with $r \in \{-1,0,1\}$, then from any configuration $(p,n_1,\ldots,n_i,\ldots,n_k)$, the machine $M$ may move to the configuration $(q,n_1,\ldots,n_i+r,\ldots,n_k)$. We may assume that counter machines never try to decrement a counter below $0$. The machine is initialized from a configuration $(q_0, 0, \ldots, 0)$, and it halts when it reaches the final state $q_f$. The machine $M$ is called \emph{reversible} if every configuration can be reached in one step from at most one other configuration. Note that a reversible counter machine need not be bijective, only injective. The classical reference for counter machines is \cite{Mi67}, although our precise definition comes from \cite{Mo96}.

There is a convenient way of converting an arbitrary counter machine into a reversible machine with only two counters, and we use this in several constructions. The following result can be extracted from the proofs in \cite{Mo96}, and we especially note that while the original results concern only deterministic machines, the following lemmas hold even for nondeterministic ones with exactly the same proofs.

\begin{lemma}[Proved as Theorem 3.1 of \cite{Mo96}]
\label{thm:Morita}
For any $k$-counter machine $M=(k,\Sigma,\delta,q_0,q_f)$ there exists a reversible $(k+2)$-counter machine $M'=(k+2,\Sigma \cup \Delta,\delta',q_0,q_f)$ such that for all $m_i, n_i, h \in \N$ and $q, p \in \Sigma$,
\[ (q,m_1,\ldots,m_k) \Rightarrow_M (p,n_1,\ldots,n_k) \]
holds if and only if there exists $\ell \in \N$ with
\[ (q,m_1,\ldots,m_k,h,0) \Rightarrow^*_{M'} (p,n_1,\ldots,n_k,\ell,0) \]
where the intermediate states of the computation are in $\Delta$. Also, all transitions of $\delta'$ whose first component is in $\Delta$ are deterministic.
\end{lemma}

\begin{lemma}[Proved as Theorem 4.1 of \cite{Mo96}]
\label{thm:Morita2}
Let $M=(k,\Sigma,\delta,q_0,q_f)$ be a reversible $k$-counter machine, and denote by $p_i$ the $i$th prime number. Then, there exists a reversible two-counter machine $M'=(2,\Sigma \cup \Delta,\delta',q_0,q_f)$ such that for all $m_i, n_i \in \N$ and $q, p \in \Sigma$,
\[ (q,m_1,\ldots,m_k) \Rightarrow_M (p,n_1,\ldots,n_k) \]
holds if and only if
\[ (q,p_1^{m_1} \cdots p_k^{m_k},0) \Rightarrow^*_{M'} (p,p_1^{n_1} \cdots p_k^{n_k},0) \]
where the intermediate states of the computation are in $\Delta$. Also, all transitions of $\delta'$ whose first component is in $\Delta$ are deterministic.
\end{lemma}

Note in particular that the countability of the set of infinite computations is preserved by the transformations.

\subsection{A few words on the figures and proofs}

As most of our results are constructions of SFTs whose configurations are infinite two-dimensional pictures with some desired properties, it should not surprise the reader that the article contains a lot of drawings. The configurations of our SFTs usually have large uniformly colored triangles and rectangles, and we specify which edges of these shapes can be glued together. We work directly with SFTs, and do not give sets of Wang tiles for any of the constructions, as implementing the rational lines with Wang tiles is cumbersome and automatizable. However, in some of our general proofs we assume that SFTs are given as sets of Wang tiles or as sets of allowed patterns of other shapes, but this will be explicitly stated. The pictures are best viewed in color, but we have also included labels for the colors in the figures when possible. The labels used are shown in Table~\ref{tab:MainColors}.

\setlength{\tabcolsep}{5pt}
\begin{table}
\caption{The labels of the main colors.}
\begin{center}
\begin{tabular}{|c||c|c|c|c|c|c|c|c|c|c|}
\hline
Color & White & Gray & Dark gray & Orange & Dark orange\\
\hline
Label & W & Gr & DGr & O & DO\\
\hline \hline
Color & Red & Dark red & Light red & Green & Dark green\\
\hline
Label & R & DR & LR & G & DG\\
\hline \hline
Color & Blue & Dark blue & Light blue & Yellow & Dark yellow\\
\hline
Label & B & DB & LB & Y & DY\\
\hline
\end{tabular}
\end{center}
\label{tab:MainColors}
\end{table}

Multiple levels of abstraction are used in figures. In Figure~\ref{fig:GridShift}, an actual partial configuration of an SFT is given. However, most figures are sofic images of the actual SFTs, leaving out some details which are explained in the text. Figure~\ref{fig:CounterComputing} is an example of a projection where most of the detail is left in. Some of the different overlapping regions are not shown (for example, the area to the left of the computation head should be colored differently from the area on its right), and only partial information about the state of the computation head is shown. On many occasions, we take this further, drawing discrete versions of rational lines as straight lines and leaving out the contours of tiles to reduce clutter, and the details are given in the text.

As is common practise, we do not go in too much detail on statements whose proof is a straightforward geometrical case analysis, and we do not show the specifics of the counter machines used in the proofs but instead just give the algorithms they execute in pseudocode.

\section{Basic results and examples}
\label{sec:Basic}

In this section, we go through some basic results on the different notions we have defined for subshifts, and give some examples of subshifts with interesting properties. The following characterizes the countability of subshifts:

\begin{lemma}[\cite{BaDuJe08}]
\label{lem:CountableRanked}
A subshift $X$ is ranked ($X^{(\lambda)} = \emptyset$ for some ordinal $\lambda$) if and only if it is countable.
\end{lemma}

We need the following lemmas about Cantor-Bendixson ranks and derivatives. The proofs are straightforward, but we give them for completeness.

\begin{lemma}
\label{lem:OrderPreserving}
A subspace of a ranked topological space is ranked. Furthermore, taking the Cantor-Bendixson rank is an order-preserving operation from ranked topological spaces to ordinals.
\end{lemma}

\begin{proof}
Clearly, $X \subset Y$ implies $X^{(\lambda)} \subset Y^{(\lambda)}$ for all ordinals $\lambda$ by transfinite induction. Because $Y$ is ranked, $Y^{(\lambda)} = \emptyset$ holds for the Cantor-Bendixson rank $\lambda$ of $Y$. Then also $X^{(\lambda)} = \emptyset$, so the Cantor-Bendixson rank of $X$ cannot be more than $\lambda$. 
\end{proof}

\begin{lemma}
\label{lem:OpenDerivative}
Let $X$ be a topological space and $Y$ an open subset of $X$. Then, for all ordinals $\lambda$, we have $X^{(\lambda)} \cap Y = Y^{(\lambda)}$.
\end{lemma}

\begin{proof}
It is clear that $Y^{(\lambda)} \subset X^{(\lambda)}$ and $Y^{(\lambda)} \subset Y$, so it is enough to show that $X^{(\lambda)} \cap Y \subset Y^{(\lambda)}$ for every ordinal $\lambda$. We prove the claim by transfinite induction. As a base case, we have $X^{(0)} \cap Y = Y = Y^{(0)}$. Now, let $\lambda = \lambda' + 1$ for some ordinal $\lambda'$. Let $x \in X^{(\lambda)} \cap Y$, and let $U$ be an open neighborhood of $x$ in $Y$. Since $U$ is also open in $X$, there exists
\[ y \in X^{(\lambda')} \cap U = Y^{(\lambda')} \cap U \]
with $y \neq x$. As for limit ordinals $\lambda$, we have $X^{(\lambda')} \cap Y \subset Y^{(\lambda')}$, for all $\lambda' < \lambda$, so
\[ X^{(\lambda)} \cap Y = \bigcap_{\lambda' < \lambda} X^{(\lambda')} \cap Y \subset \bigcap_{\lambda' < \lambda} Y^{(\lambda')} = Y^{(\lambda)}, \]
which concludes the proof. 
\end{proof}

\subsection{The one-dimensional case}

In this section we look at one-dimensional (not necessarily countable) sofic shifts to illustrate the difference between the one-dimensional and the multidimensional case. It is more natural to consider sofic shifts instead of SFTs, as sofic shifts are closed under the derivative operation by Proposition~\ref{prop:1DSoficNiceness}, but the derivative of an SFT can be a proper sofic shift, by Example~\ref{ex:SFTDerivativeSofic}.

One-dimensional sofic shifts have a useful and well-known characterization, Lemma~\ref{lem:SoficIffContexts}, in terms of the different contexts of words that appear in them. We give the characterization without proof, but for example, it easily follows from Theorem 3.2.10 of \cite{LiMa95}.

\begin{definition}
Let $X \subset S^\Z$ be a subshift with $\B_1(X) = S$. The \emph{extender set} of a word $v \in S^*$ in $X$ is $E_X(v) = \{ (w,w') \in (S^*)^2 \;|\; wvw' \sqsubset X \}$.
\end{definition}

The assumption $\B_1(X) = S$ is only for definiteness, as we can then discuss the extender sets of words in a subshift $X$ without specifying $S$.

\begin{lemma}
\label{lem:SoficIffContexts}
A one-dimensional subshift is sofic if and only if it has a finite number of different extender sets.
\end{lemma}

We now relate the extender sets of words in a subshift and its derivative.

\begin{lemma}
\label{lemma:SameContexts}
For a subshift $X \subset S^\Z$,
\[ E_X(u) = E_X(v) \implies E_{X^{(1)}}(u) = E_{X^{(1)}}(v). \]
\end{lemma}

\begin{proof}
Let $E_X(u) = E_X(v)$, and suppose that $(w, w') \in E_{X^{(1)}}(u) - E_{X^{(1)}}(v)$. Then $wvw' \not\sqsubset X^{(1)}$, so the set of points $x \in X$ with $x_{[0,|wvw'|-1]} = wvw'$ is finite. But since $E_X(u) = E_X(v)$, these are in a bijective correspondence with the points $y$ such that $y_{[0,|wuw'|-1]} = wuw'$, which implies that $wuw' \not\sqsubset X^{(1)}$, a contradiction. 
\end{proof}

This implies that the number of different extender sets cannot increase in the derivative, so the two previous lemmas give the following:

\begin{proposition}
\label{prop:1DSoficNiceness}
The derivative of a one-dimensional sofic shift is sofic.
\end{proposition}

A simple further analysis proves the following result.

\begin{proposition}
\label{prop:SoficFinite}
All one-dimensional sofic shifts have finite rank. More explicitly, the rank of a one-dimensional sofic shift is bounded by the number of different extender sets in it.
\end{proposition}

\begin{proof}
Let $X$ be sofic, and let $k$ be the number of different extender sets in $X$. If $X \neq X^{(1)}$, then necessarily $X \neq S^\Z$, so we may choose $u \notin \B(X)$. Further, choose $v \in \B(X) - \B(X^{(1)})$. Now, $E_X(u) \neq E_X(v)$, but $E_{X^{(1)}}(u) = \emptyset = E_{X^{(1)}}(v)$. By Lemma~\ref{lemma:SameContexts}, $X^{(1)}$ has at most $k - 1$ different extender sets. By induction, $X^{(i)}$ has at most $k - i$ different extender sets, and it is then clear that $X^{(i)} = X^{(i+1)}$ for some $i \leq k$. 
\end{proof}

We now show an example of an SFT whose derivative is proper sofic (sofic but not an SFT). Compare this with the SFT constructed in Theorem~\ref{thm:DerivativesCanCompute}, which is in particular a two-dimensional SFT whose derivative is not an SFT.

\begin{example}
\label{ex:SFTDerivativeSofic}
The subshift $X = \B^{-1}(a^*b^*c^*)$ is an SFT. The isolated points in it are all configurations where both $a$ and $c$ occur. Thus, $X^{(1)} = \B^{-1}(a^*b^* + b^*c^*)$, which is a proper sofic shift.
\end{example}

We show the following structure lemma of countable sofic shifts (and in particular SFTs), both to show another difference between the one-dimensional and multidimensional case and to make the restrictiveness of the \PProperty more explicit. The version we give is Lemma~1 in \cite{SaTo12c} (see also Lemma~4.8 of \cite{PaSc10}).

\begin{lemma}
\label{lem:UniqueParsing}
Let $X \subset S^\Z$ be a countable sofic shift. Then there exists a finite set $T$ of tuples of words in $\B(X)$ such that every configuration $x \in X$ is representable as
\[ x = \INF u_0 v_1 u_1^{n_1} \cdots u_{m-1}^{n_{m-1}} v_m u_m^\infty \]
for a unique $t = (u_0,\ldots,u_m,v_1,\ldots,v_m) \in T$. In particular, $X$ has only finitely many periodic configurations.
\end{lemma}

Finally, we show that a countable one-dimensional sofic shift has a rather uninteresting subpattern poset.

\begin{lemma}
\label{lem:1DInfThenPer}
Let $X \subset S^\Z$ be a sofic shift. If $w \in S^*$ occurs infinitely many times in a configuration $x \in X$, then $X$ has a periodic configuration in which $w$ occurs.
\end{lemma}

\begin{lemma}
\label{lem:1DBoringPoset}
If $X \subset S^\Z$ is a countable sofic shift, then $x < y$ for $x, y \in X$ if and only if $x$ is a periodic point and $y$ is an eventually periodic point whose left or right tail is equal to the corresponding tail of $x$, up to a power of the shift.
\end{lemma}

\begin{proof}
If $x, y \in X$ and $x < y$, then $x \neq y$ and $y$ shares arbitrarily long subwords with $x$. Since $x$ and $y$ are in $X$, it follows from Lemma~\ref{lem:UniqueParsing} that they are eventually periodic. If $x$ is actually periodic, then as one of the tails of $y$ must contain arbitrarily long subwords of $x$, the tail must actually be equal to $x$ up to a shift by the theorem of Fine and Wilf \cite{FiWi65}. If $x$ is not periodic, let $w_i = x_{[-i, i]}$. Since $w_i \sqsubset y$ but $x \neq y$, $y$ must contain infinitely many copies of $w_i$ for all $i$. By Lemma~\ref{lem:1DInfThenPer}, all of the $w_i$ then occur in periodic points, and if $x$ is itself not periodic, we obtain infinitely many periodic points in $X$, a contradiction with Lemma~\ref{lem:UniqueParsing}.
\end{proof}

\subsection{Determinism, countability and the \PProperty}

Throughout the later sections, we discuss the properties of determinism, countability and the \PProperty for an SFT. We show by the examples listed in Table~\ref{fig:Combinations} that essentially all combinations of the three are possible. We do not know examples of uncountable SFTs with the \PProperty in every direction. Thus, in the examples, we are satisfied with the downward \PProperty for uncountable SFTs, especially as such SFTs are not the main topic of this article. Similarly, as it is known that only periodic SFTs are deterministic in every direction \cite{BoLi97}, we usually give only one direction of determinism. The nondeterministic SFTs we present are deterministic in no direction. Example~\ref{ex:Grid} will be used later, in the proof of Theorem~\ref{thm:ManyPosets}, and Example~\ref{ex:YYN} is essentially a simplified version of Construction~\ref{con:CounterMachine}. Otherwise, the results of this section serve mainly as an introduction to the different concepts of this article, and are not used in the proofs of the main results.

\begin{example}
\label{ex:Grid}
We give an example of a countable SFT which does not have the \PProperty in any direction but is deterministic in all but the six directions $(1,0), (-1,0), (0,1), (0,-1), (1,-1)$ and $(-1,1)$. This is the grid shift from \cite{SaTo12b}. It consists of infinite horizontal and vertical lines that divide each configuration into rectangles. The rectangles are forced to be squares by coloring their northwest half differently from the southeast half, using a diagonal signal as a separator. The grid shift is over the alphabet $\{0, 1, 2\}$ and is defined by the allowed patterns of size $2 \times 2$ which occur in Figure~\ref{fig:GridShift}.
\begin{figure}[ht]
\begin{center}
\begin{tikzpicture}[scale=0.5]

\fill[color=red!50!white] (0,0) rectangle +(1,1);
\node () at (0.5,0.5) {$2$};
\node () at (1.5,0.5) {$0$};
\node () at (2.5,0.5) {$0$};
\node () at (3.5,0.5) {$0$};
\node () at (4.5,0.5) {$0$};
\fill[color=red!50!white] (5,0) rectangle +(1,1);
\node () at (5.5,0.5) {$2$};
\node () at (6.5,0.5) {$0$};
\node () at (7.5,0.5) {$0$};
\node () at (8.5,0.5) {$0$};
\node () at (9.5,0.5) {$0$};
\fill[color=gray!50!white] (0,1) rectangle +(1,1);
\node () at (0.5,1.5) {$1$};
\fill[color=red!50!white] (1,1) rectangle +(1,1);
\node () at (1.5,1.5) {$2$};
\node () at (2.5,1.5) {$0$};
\node () at (3.5,1.5) {$0$};
\node () at (4.5,1.5) {$0$};
\fill[color=gray!50!white] (5,1) rectangle +(1,1);
\node () at (5.5,1.5) {$1$};
\fill[color=red!50!white] (6,1) rectangle +(1,1);
\node () at (6.5,1.5) {$2$};
\node () at (7.5,1.5) {$0$};
\node () at (8.5,1.5) {$0$};
\node () at (9.5,1.5) {$0$};
\fill[color=gray!50!white] (0,2) rectangle +(1,1);
\node () at (0.5,2.5) {$1$};
\fill[color=gray!50!white] (1,2) rectangle +(1,1);
\node () at (1.5,2.5) {$1$};
\fill[color=red!50!white] (2,2) rectangle +(1,1);
\node () at (2.5,2.5) {$2$};
\node () at (3.5,2.5) {$0$};
\node () at (4.5,2.5) {$0$};
\fill[color=gray!50!white] (5,2) rectangle +(1,1);
\node () at (5.5,2.5) {$1$};
\fill[color=gray!50!white] (6,2) rectangle +(1,1);
\node () at (6.5,2.5) {$1$};
\fill[color=red!50!white] (7,2) rectangle +(1,1);
\node () at (7.5,2.5) {$2$};
\node () at (8.5,2.5) {$0$};
\node () at (9.5,2.5) {$0$};
\fill[color=gray!50!white] (0,3) rectangle +(1,1);
\node () at (0.5,3.5) {$1$};
\fill[color=gray!50!white] (1,3) rectangle +(1,1);
\node () at (1.5,3.5) {$1$};
\fill[color=gray!50!white] (2,3) rectangle +(1,1);
\node () at (2.5,3.5) {$1$};
\fill[color=red!50!white] (3,3) rectangle +(1,1);
\node () at (3.5,3.5) {$2$};
\node () at (4.5,3.5) {$0$};
\fill[color=gray!50!white] (5,3) rectangle +(1,1);
\node () at (5.5,3.5) {$1$};
\fill[color=gray!50!white] (6,3) rectangle +(1,1);
\node () at (6.5,3.5) {$1$};
\fill[color=gray!50!white] (7,3) rectangle +(1,1);
\node () at (7.5,3.5) {$1$};
\fill[color=red!50!white] (8,3) rectangle +(1,1);
\node () at (8.5,3.5) {$2$};
\node () at (9.5,3.5) {$0$};
\fill[color=gray!50!white] (0,4) rectangle +(1,1);
\node () at (0.5,4.5) {$1$};
\fill[color=gray!50!white] (1,4) rectangle +(1,1);
\node () at (1.5,4.5) {$1$};
\fill[color=gray!50!white] (2,4) rectangle +(1,1);
\node () at (2.5,4.5) {$1$};
\fill[color=gray!50!white] (3,4) rectangle +(1,1);
\node () at (3.5,4.5) {$1$};
\fill[color=red!50!white] (4,4) rectangle +(1,1);
\node () at (4.5,4.5) {$2$};
\fill[color=gray!50!white] (5,4) rectangle +(1,1);
\node () at (5.5,4.5) {$1$};
\fill[color=gray!50!white] (6,4) rectangle +(1,1);
\node () at (6.5,4.5) {$1$};
\fill[color=gray!50!white] (7,4) rectangle +(1,1);
\node () at (7.5,4.5) {$1$};
\fill[color=gray!50!white] (8,4) rectangle +(1,1);
\node () at (8.5,4.5) {$1$};
\fill[color=red!50!white] (9,4) rectangle +(1,1);
\node () at (9.5,4.5) {$2$};
\fill[color=red!50!white] (0,5) rectangle +(1,1);
\node () at (0.5,5.5) {$2$};
\node () at (1.5,5.5) {$0$};
\node () at (2.5,5.5) {$0$};
\node () at (3.5,5.5) {$0$};
\node () at (4.5,5.5) {$0$};
\fill[color=red!50!white] (5,5) rectangle +(1,1);
\node () at (5.5,5.5) {$2$};
\node () at (6.5,5.5) {$0$};
\node () at (7.5,5.5) {$0$};
\node () at (8.5,5.5) {$0$};
\node () at (9.5,5.5) {$0$};
\fill[color=gray!50!white] (0,6) rectangle +(1,1);
\node () at (0.5,6.5) {$1$};
\fill[color=red!50!white] (1,6) rectangle +(1,1);
\node () at (1.5,6.5) {$2$};
\node () at (2.5,6.5) {$0$};
\node () at (3.5,6.5) {$0$};
\node () at (4.5,6.5) {$0$};
\fill[color=gray!50!white] (5,6) rectangle +(1,1);
\node () at (5.5,6.5) {$1$};
\fill[color=red!50!white] (6,6) rectangle +(1,1);
\node () at (6.5,6.5) {$2$};
\node () at (7.5,6.5) {$0$};
\node () at (8.5,6.5) {$0$};
\node () at (9.5,6.5) {$0$};
\fill[color=gray!50!white] (0,7) rectangle +(1,1);
\node () at (0.5,7.5) {$1$};
\fill[color=gray!50!white] (1,7) rectangle +(1,1);
\node () at (1.5,7.5) {$1$};
\fill[color=red!50!white] (2,7) rectangle +(1,1);
\node () at (2.5,7.5) {$2$};
\node () at (3.5,7.5) {$0$};
\node () at (4.5,7.5) {$0$};
\fill[color=gray!50!white] (5,7) rectangle +(1,1);
\node () at (5.5,7.5) {$1$};
\fill[color=gray!50!white] (6,7) rectangle +(1,1);
\node () at (6.5,7.5) {$1$};
\fill[color=red!50!white] (7,7) rectangle +(1,1);
\node () at (7.5,7.5) {$2$};
\node () at (8.5,7.5) {$0$};
\node () at (9.5,7.5) {$0$};
\fill[color=gray!50!white] (0,8) rectangle +(1,1);
\node () at (0.5,8.5) {$1$};
\fill[color=gray!50!white] (1,8) rectangle +(1,1);
\node () at (1.5,8.5) {$1$};
\fill[color=gray!50!white] (2,8) rectangle +(1,1);
\node () at (2.5,8.5) {$1$};
\fill[color=red!50!white] (3,8) rectangle +(1,1);
\node () at (3.5,8.5) {$2$};
\node () at (4.5,8.5) {$0$};
\fill[color=gray!50!white] (5,8) rectangle +(1,1);
\node () at (5.5,8.5) {$1$};
\fill[color=gray!50!white] (6,8) rectangle +(1,1);
\node () at (6.5,8.5) {$1$};
\fill[color=gray!50!white] (7,8) rectangle +(1,1);
\node () at (7.5,8.5) {$1$};
\fill[color=red!50!white] (8,8) rectangle +(1,1);
\node () at (8.5,8.5) {$2$};
\node () at (9.5,8.5) {$0$};
\fill[color=gray!50!white] (0,9) rectangle +(1,1);
\node () at (0.5,9.5) {$1$};
\fill[color=gray!50!white] (1,9) rectangle +(1,1);
\node () at (1.5,9.5) {$1$};
\fill[color=gray!50!white] (2,9) rectangle +(1,1);
\node () at (2.5,9.5) {$1$};
\fill[color=gray!50!white] (3,9) rectangle +(1,1);
\node () at (3.5,9.5) {$1$};
\fill[color=red!50!white] (4,9) rectangle +(1,1);
\node () at (4.5,9.5) {$2$};
\fill[color=gray!50!white] (5,9) rectangle +(1,1);
\node () at (5.5,9.5) {$1$};
\fill[color=gray!50!white] (6,9) rectangle +(1,1);
\node () at (6.5,9.5) {$1$};
\fill[color=gray!50!white] (7,9) rectangle +(1,1);
\node () at (7.5,9.5) {$1$};
\fill[color=gray!50!white] (8,9) rectangle +(1,1);
\node () at (8.5,9.5) {$1$};
\fill[color=red!50!white] (9,9) rectangle +(1,1);
\node () at (9.5,9.5) {$2$};
\draw (0,0) grid (10,10);

\end{tikzpicture}
\end{center}
\caption{A partial configuration of the grid shift containing exactly the allowed $2 \times 2$ patterns.}
\label{fig:GridShift}
\end{figure}
\end{example}

\begin{example}
\label{ex:Nondeterministic}
We give an example of a countable SFT which is deterministic in no direction, but has the \PProperty in every direction. First, consider the SFT over $\{0, 1\}$ defined by the allowed patterns of size $2 \times 2$ which occur in Figure~\ref{fig:QuarterShiftPic}:
\begin{figure}[ht]
\begin{center}
\begin{tikzpicture}[scale=0.5]

\fill[color=gray!50!white] (0,0) rectangle +(1,1);
\node () at (0.5,0.5) {$1$};
\fill[color=gray!50!white] (1,0) rectangle +(1,1);
\node () at (1.5,0.5) {$1$};
\fill[color=gray!50!white] (2,0) rectangle +(1,1);
\node () at (2.5,0.5) {$1$};
\fill[color=gray!50!white] (3,0) rectangle +(1,1);
\node () at (3.5,0.5) {$1$};
\fill[color=gray!50!white] (4,0) rectangle +(1,1);
\node () at (4.5,0.5) {$1$};
\fill[color=gray!50!white] (5,0) rectangle +(1,1);
\node () at (5.5,0.5) {$1$};
\fill[color=gray!50!white] (6,0) rectangle +(1,1);
\node () at (6.5,0.5) {$1$};
\fill[color=gray!50!white] (7,0) rectangle +(1,1);
\node () at (7.5,0.5) {$1$};
\fill[color=gray!50!white] (0,1) rectangle +(1,1);
\node () at (0.5,1.5) {$1$};
\fill[color=gray!50!white] (1,1) rectangle +(1,1);
\node () at (1.5,1.5) {$1$};
\fill[color=gray!50!white] (2,1) rectangle +(1,1);
\node () at (2.5,1.5) {$1$};
\fill[color=gray!50!white] (3,1) rectangle +(1,1);
\node () at (3.5,1.5) {$1$};
\fill[color=gray!50!white] (4,1) rectangle +(1,1);
\node () at (4.5,1.5) {$1$};
\fill[color=gray!50!white] (5,1) rectangle +(1,1);
\node () at (5.5,1.5) {$1$};
\fill[color=gray!50!white] (6,1) rectangle +(1,1);
\node () at (6.5,1.5) {$1$};
\fill[color=gray!50!white] (7,1) rectangle +(1,1);
\node () at (7.5,1.5) {$1$};
\fill[color=gray!50!white] (0,2) rectangle +(1,1);
\node () at (0.5,2.5) {$1$};
\fill[color=gray!50!white] (1,2) rectangle +(1,1);
\node () at (1.5,2.5) {$1$};
\fill[color=gray!50!white] (2,2) rectangle +(1,1);
\node () at (2.5,2.5) {$1$};
\fill[color=gray!50!white] (3,2) rectangle +(1,1);
\node () at (3.5,2.5) {$1$};
\fill[color=gray!50!white] (4,2) rectangle +(1,1);
\node () at (4.5,2.5) {$1$};
\fill[color=gray!50!white] (5,2) rectangle +(1,1);
\node () at (5.5,2.5) {$1$};
\fill[color=gray!50!white] (6,2) rectangle +(1,1);
\node () at (6.5,2.5) {$1$};
\fill[color=gray!50!white] (7,2) rectangle +(1,1);
\node () at (7.5,2.5) {$1$};
\fill[color=gray!50!white] (0,3) rectangle +(1,1);
\node () at (0.5,3.5) {$1$};
\fill[color=gray!50!white] (1,3) rectangle +(1,1);
\node () at (1.5,3.5) {$1$};
\fill[color=gray!50!white] (2,3) rectangle +(1,1);
\node () at (2.5,3.5) {$1$};
\fill[color=gray!50!white] (3,3) rectangle +(1,1);
\node () at (3.5,3.5) {$1$};
\fill[color=gray!50!white] (4,3) rectangle +(1,1);
\node () at (4.5,3.5) {$1$};
\fill[color=gray!50!white] (5,3) rectangle +(1,1);
\node () at (5.5,3.5) {$1$};
\fill[color=gray!50!white] (6,3) rectangle +(1,1);
\node () at (6.5,3.5) {$1$};
\fill[color=gray!50!white] (7,3) rectangle +(1,1);
\node () at (7.5,3.5) {$1$};
\node () at (0.5,4.5) {$0$};
\node () at (1.5,4.5) {$0$};
\node () at (2.5,4.5) {$0$};
\node () at (3.5,4.5) {$0$};
\fill[color=red!50!white] (4,4) rectangle +(1,1);
\node () at (4.5,4.5) {$2$};
\fill[color=red!50!white] (5,4) rectangle +(1,1);
\node () at (5.5,4.5) {$2$};
\fill[color=red!50!white] (6,4) rectangle +(1,1);
\node () at (6.5,4.5) {$2$};
\fill[color=red!50!white] (7,4) rectangle +(1,1);
\node () at (7.5,4.5) {$2$};
\node () at (0.5,5.5) {$0$};
\node () at (1.5,5.5) {$0$};
\node () at (2.5,5.5) {$0$};
\node () at (3.5,5.5) {$0$};
\fill[color=red!50!white] (4,5) rectangle +(1,1);
\node () at (4.5,5.5) {$2$};
\fill[color=red!50!white] (5,5) rectangle +(1,1);
\node () at (5.5,5.5) {$2$};
\fill[color=red!50!white] (6,5) rectangle +(1,1);
\node () at (6.5,5.5) {$2$};
\fill[color=red!50!white] (7,5) rectangle +(1,1);
\node () at (7.5,5.5) {$2$};
\node () at (0.5,6.5) {$0$};
\node () at (1.5,6.5) {$0$};
\node () at (2.5,6.5) {$0$};
\node () at (3.5,6.5) {$0$};
\fill[color=red!50!white] (4,6) rectangle +(1,1);
\node () at (4.5,6.5) {$2$};
\fill[color=red!50!white] (5,6) rectangle +(1,1);
\node () at (5.5,6.5) {$2$};
\fill[color=red!50!white] (6,6) rectangle +(1,1);
\node () at (6.5,6.5) {$2$};
\fill[color=red!50!white] (7,6) rectangle +(1,1);
\node () at (7.5,6.5) {$2$};
\node () at (0.5,7.5) {$0$};
\node () at (1.5,7.5) {$0$};
\node () at (2.5,7.5) {$0$};
\node () at (3.5,7.5) {$0$};
\fill[color=red!50!white] (4,7) rectangle +(1,1);
\node () at (4.5,7.5) {$2$};
\fill[color=red!50!white] (5,7) rectangle +(1,1);
\node () at (5.5,7.5) {$2$};
\fill[color=red!50!white] (6,7) rectangle +(1,1);
\node () at (6.5,7.5) {$2$};
\fill[color=red!50!white] (7,7) rectangle +(1,1);
\node () at (7.5,7.5) {$2$};
\draw (0,0) grid (8,8);

\end{tikzpicture}
\end{center}
\caption{A partial configuration of the quarter plane shift containing exactly the allowed $2 \times 2$ patterns.}
\label{fig:QuarterShiftPic}
\end{figure}
This SFT is not deterministic in any direction $(x, y) \in \Z^2$ with $y \geq 0$. Analogously, $A_{\frac{\pi}{2}}^2(X)$ is not deterministic in any direction $(x, y) \in \Z^2$ with $y \leq 0$, so $X \times A_{\frac{\pi}{2}}^2(X)$ is not deterministic in any direction. However, it is countable and has the \PProperty in every direction.
\end{example}

\begin{example}
\label{ex:YYN}
We give an example of a downward deterministic uncountable SFT with the \PProperty. An indirect way to find such an example is to use Construction~\ref{con:CounterMachine}: Let $M$ be a nondeterministic counter machine with a single counter that it can choose to increment or preserve on each step, use Lemma~\ref{thm:Morita} to make it reversible, and plug the resulting machine in Construction~\ref{con:CounterMachine}. The resulting SFT is then deterministic in direction $(0, -1)$ and has the \PProperty, but it is uncountable as the choice of whether the counter is incremented or not at each step is visible in the configuration.

We also give a direct construction, which can be regarded as a `broken' version of Construction~\ref{con:CounterMachine}: There, a bouncing signal moves upward in an expanding cone, and its movement is extendably north deterministic. We did not make this choice of direction only because the counter machine intuitively computes in this direction. This example reverses the direction of determinism, making the SFT extendably south deterministic, and illustrates what can go wrong with this choice.

We construct the two-dimensional SFT $X$ by first defining a one-dimensional SFT $Y$ satisfying $\PS(X) \subset Y$ and then defining $X$ as the limit spacetime diagrams of a cellular automaton $f$ on $Y$, running downward. Denote $D = \{\leftarrow, \rightarrow\}$. The subshift $Y$ is defined by
\[ Y = \B^{-1}(0^*\ell^* D r^*1^*), \]
and $f$ is the cellular automaton that, denoting $\phi(a, d, b) = \INF 0 .\ell^a d r^b 1 \INF$ where $d \in D$, maps
\begin{align*}
f(\phi(a, \leftarrow, b)) = &\left\{\begin{array}{ll}
\phi(a - 2, \leftarrow, b + 1), & \mbox{if } a \geq 2, \\
\phi(0, \rightarrow, 0), & \mbox{if } a < 2 \wedge b = 0, \\
\phi(a, \rightarrow, b - 1), & \mbox{if } a < 2 \wedge b > 0
\end{array}\right. \\
f(\phi(a, \rightarrow, b)) = &\left\{\begin{array}{ll}
\phi(a + 2, \rightarrow, b - 3), & \mbox{if } b \geq 3, \\
\phi(a, \leftarrow, 0), & \mbox{if } b < 3
\end{array}\right.
\end{align*}
Note here that the dot in the definition of $\phi(a, d, b)$ marks the origin, so that $f$ does not move the left border of the segment $\ell^a d r^b$. The automaton $f$ so defined has minimal radius $2$.

The resulting SFT $X$ is clearly downward deterministic and has the downward \PProperty. It is also uncountable: every pattern $\phi(a, d, b)$ with $a, b \geq 3$ has a preimage of the same form, and following a preimage chain eventually leads to a pattern $\phi(a, d, b)$ with $b = 0$, which has multiple preimages.
\end{example}

\setlength{\tabcolsep}{3pt}
\renewcommand{\arraystretch}{1.3}
\begin{table}
\caption{The 8 different combinations of the properties determinism, countability and the \PProperty. The `yes' cases are emphasized.}
\begin{center}
\begin{tabular}{|c|c|c|c|}
\hline
Example & Deterministic & \CapPProperty & Countable \\
\hline
$\{0,1\}^{\Z^d}$ & \red{no} & \red{in no direction} & \red{no} \\
\hline
Example~\ref{ex:Grid} $\; \times$ Example~\ref{ex:Nondeterministic} & \red{no} & \red{in no direction} & \green{yes} \\
\hline
Example~\ref{ex:Nondeterministic} $\times \; \{x \;|\; x = \sigma^{(1,0)}(x)\}$ & \red{no} & \green{downward} & \red{no} \\
\hline
Example~\ref{ex:Nondeterministic} & \red{no} & \green{in every direction} & \green{yes} \\
\hline
$\{ x \;|\; x \in \{\sigma^{(1,0)}(x), \sigma^{(1,1)}(x)\} \}$ & \green{downward} & \red{in no direction} & \red{no} \\
\hline
Example~\ref{ex:Grid} & \green{southwest} & \red{in no direction} & \green{yes} \\
\hline
Example~\ref{ex:YYN} & \green{downward} & \green{downward} & \red{no} \\
\hline
$\{0\}^{\Z^2}$ & \green{in every direction} & \green{in every direction} & \green{yes} \\
\hline
\end{tabular}
\end{center}
\label{fig:Combinations}
\end{table}

An interesting property of SFTs that are both deterministic and have the \PProperty is that they, in a sense, come from the type of cellular automata studied in \cite{SaTo12c}. We state the following theorems only for downward determinism, as one can apply a transformation in $GL_2(\Z)$ to rotate the direction of determinism.

\begin{proposition}
Let $X \subset S^{\Z^2}$ be a deterministic countable SFT with the \PProperty. Then there is a cellular automaton $f$ on a countable one-dimensional SFT $Y \subset (S \dot\cup \{\#\})^\Z$ such that $X$ is conjugate to the set of limit spacetime diagrams of $f$ not containing $\#$.
\end{proposition}

The need for $\#$ comes from the fact that our definition of determinism allows for SFTs where half-planes can be locally legal, while not having legal extensions (so a half-plane can have 0 or 1 extensions to the full plane). If we assume all half-planes extend to a unique configuration, we have the following more natural result.

\begin{proposition}
\label{prop:ConnectionLOL}
Let $X \subset S^{\Z^2}$ be an extendably deterministic countable SFT with the \PProperty. Then there is a cellular automaton $f$ on a countable one-dimensional SFT $Y \subset S$ such that $X$ is conjugate to the set of limit spacetime diagrams of $f$. Furthermore, the local function of $f$ is computable from the forbidden patterns of $X$.
\end{proposition}

A similar result is true for SFTs without the \PProperty, but with $Y$ not necessarily countable. The constructions in this paper are deterministic only in the weaker sense, and thus do not directly come from cellular automata on countable SFTs. Furthermore, even if one restricts to the limit spacetime diagrams where $\#$ does not occur, the direction of determinism in our constructions is usually not a very interesting one: Most of our constructions are cones extending northeast from a seed pattern, and everything of interest happens within the cone. The direction of determinism, however, usually has a negative y-component. That is, the cellular automaton can only recreate a finite initial pattern from the infinite tail of the cone.

While our examples imply that all the combinations of determinism, countability and the \PProperty are possible, there are some interesting connections between them, as for instance, Lemma~\ref{lem:RProperty} holds for countable SFTs with the \PProperty and deterministic SFTs with the \PProperty for slightly different reasons, but not for general SFTs with the \PProperty.

\section{Cantor-Bendixson ranks and complexity of derivatives}

The two main constructions of this section concern the computational complexity of the $k$th derivative, and the upper bound of all attainable Cantor-Bendixson ranks.

Theorem~\ref{thm:DerivativesCanCompute} shows that there exists an SFT whose derivatives of order less than $\omega$ have maximal possible computational complexity. This is essentially Theorem~1 in \cite{SaTo12b}, although we make the \PProperty explicit, and determinize the construction.

Theorem~\ref{thm:BigRanks} is a slight improvement on Theorem~4.3 in \cite{JeVa11}, which states that Cantor-Bendixson ranks of countable two-dimensional SFTs are cofinal in the Cantor-Bendixson ranks of countable $\PI^0_1$ sets. More specifically, the theorem of \cite{JeVa11} states that for every countable $\PI^0_1$ set with Cantor-Bendixson rank $\lambda$, there is a countable SFT with rank $\lambda + 11$. Our addition is that we obtain the high ranks with deterministic SFTs having the \PProperty, and we lower the `simulation overhead' from $11$ to $4$. The construction in Theorem~4.3 in \cite{JeVa11} could be made deterministic, but counter machines are essential for the \PProperty. Furthermore, the \PProperty forces the binary track (which contains a point of the $\PI^0_1$ set) used in the proof of Theorem~\ref{thm:BigRanks} to be `simulated in software', while \cite{JeVa11} uses an actual tape.

Both constructions are subcases of a more general method of embedding computation in a countable two-dimensional SFT using counter machines, which we present before the actual results. The construction is used again later, in connection with subpattern posets. 

\begin{construction}[Embedding computations in a countable SFT with the \PProperty]
\label{con:CounterMachine}
Suppose we are given a nondeterministic counter machine $M = (k,\Sigma,\delta,q_0,q_f)$ with
\[ \delta \subset (\Sigma \times [1,k] \times \{Z,P\} \times \Sigma) \cup (\Sigma \times [1,k] \times \{-1,0,1\} \times \Sigma). \]
Further suppose that the set of infinite computations of $M$ is countable; note that the set of halting computations is countable for every counter machine, as such computations are finite. We construct a countable SFT $X$ with the \PProperty whose configurations correspond to computation histories of $M$ in a concrete way. A single row of $X$ consists of a computation zone, which contains segments of special counter symbols whose lengths encode the counter values, together with a `zig zag head' that holds a state of $M$. The head sweeps back and forth, updating its state and the counter values, and the computation area increases in size, so that $X$ will contain an infinite computation cone extending upwards. If $M$ is reversible, then $X$ will be south deterministic. However, as the counter machine is running upward, we describe how rows evolve bottom-up.

We define the SFT $X$ by first defining a set of legal rows $Y$, and then adding constraints that state which configurations may occur on top of each other (which may further reduce the set of legal rows). That is, we first define a subshift $Y$ such that $\PS(X) \subset Y$, and then describe a relation $R \subset Y^2$ which is a subSFT of $Y^2$ such that $X = \{ x \in S^{\Z^2} \;|\; \forall i: (x_i, x_{i+1}) \in R \}$.

We define
\[ \phi((m_\ell, h, m_r), (n_1, \ldots, n_k)) = \INF 0. (\ell^{m_\ell} h r^{m_r} 1 \INF \times C_1^{n_1} 1 \INF \times \cdots \times C_k^{n_k} 1 \INF), \]
where $m_\ell, m_r, n_i \in \N$, $h \in \{\leftz, \leftp, \rightarrow\} \times \Sigma \times [1,k] \times \{-1, 0, 1\}$. The interpretation is that $n_i$ is the value in the $i$th counter, the width of the computation area of the current row is $n_\ell + 1 + n_r$, and the zig zag head is represented by $h$. The interpretation of $h = (\rightarrow, q, i, d)$ is that the zig zag head is moving to the right in state $q$ in order to reach the right border of the computation zone. The interpretation of $h = (\stackrel{b}{\leftarrow}, s, i, d)$ is that the zig zag head is moving left, and is trying to increment the value of counter $i$ by $b \cdot d$. The subshift $Y$ is the orbit closure of the image of $\phi$ and configurations
\[ \INF 0. (1 \INF \times C_1^0 1 \INF \times \cdots \times C_k^0 1 \INF), \]
where computation has not started yet.

For the configurations $x = \INF 0. (1 \INF \times C_1^0 1 \INF \times \cdots \times C_k^0 1 \INF)$ where computation has not started yet, we allow $(x, x) \in R$ and $(x, y) \in R$ for
\[ y = \phi((0, h, 0), (0, \ldots, 0)), \]
where $h = (\leftz, q_0, 1, 0)$. The choice of the content of $h$ here (other than the initial state $q_0$) is somewhat arbitrary.

Configurations $\phi((m_\ell, h, m_r), (n_1, \ldots, n_k))$ are followed by configurations of the same form. Note in particular that the left border of the computation area does not move. If $h = (\rightarrow, q, i, d)$, the zig zag head is moving to the right, so the values $(m_\ell, h, m_r)$ are updated to $(m_\ell+2, h, m_r-1)$ if $m_r \geq 1$, and to $(m_\ell, (\leftp, q, i, d), m_r+1)$ if $m_r = 0$. The counter values are not changed.

If $h = (\stackrel{b}{\leftarrow}, q, i, d)$, the zig zag head is moving to the left, so the values $(m_\ell, h, m_r)$ are updated to $(m_\ell-1, h', m_r+2)$ if $m_{\ell} \geq 1$, and $(m_\ell, h'', m_r+1)$ if $m_\ell = b = 0$. The case $m_\ell = 0, b = 1$ leads to an error. Here, $h' = h$ if $n_i \neq m_\ell - 1$, and $h' = (\leftz, q, i, d)$ otherwise. If $n_i = m_\ell - 1$, the value $n_i$ is also updated to $n_i + b \cdot d$. Thus the bit $b$ prevents the zig zag head from updating the same counter multiple times. The element $h''$ can be chosen as $(\rightarrow, p, 1, 0)$ if $(q,j,Z,p) \in \delta$ and $n_j = 0$, or $(q,j,P,p) \in \delta$ and $n_j > 0$. Also, $h''$ can be chosen as $(\rightarrow, p, j, d')$ if $(q,j,d',p) \in \delta$ with $d' \in \{-1,0,1\}$.

We remark here that, in contrast to Example~\ref{ex:YYN}, the movement of the head is north deterministic, so the possible uncountability of $X$ can only be the result of $M$ having uncountably many computation histories.

The beginning of a computation is shown in Figure~\ref{fig:CounterComputing}.

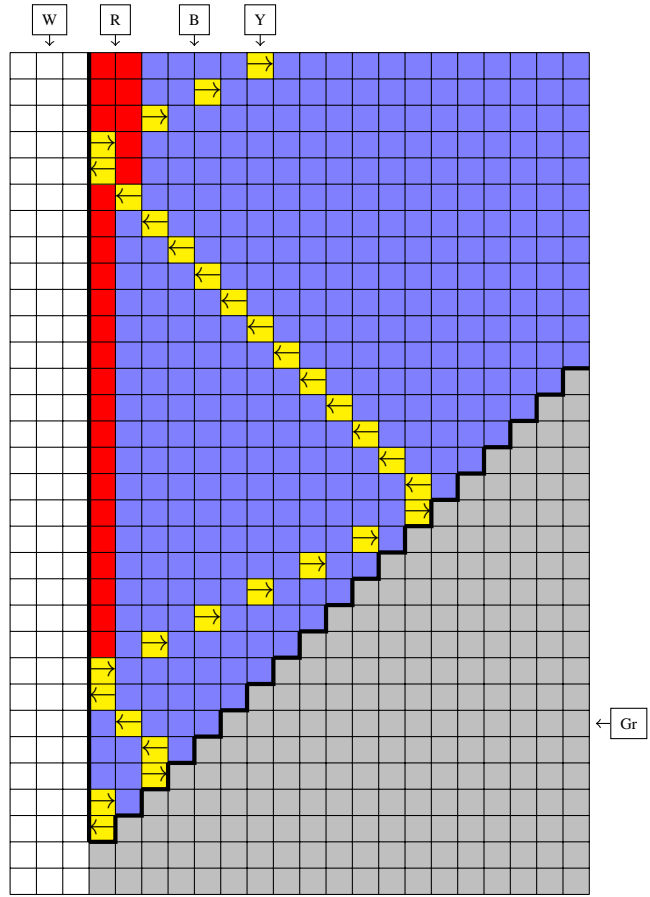
\begin{figure}[ht]
\begin{center}
\begin{tikzpicture}[scale=0.35]

\fill[color=blue!50!white] (4,12) rectangle (10,21);
\fill[color=yellow] (12,18) rectangle (13,19);
\node () at (12.5,18.5) {$\leftarrow$};
\fill[color=blue!50!white] (8,23) rectangle (9,32);
\fill[color=gray!50!white] (7,0) rectangle (22,6);
\fill[color=blue!50!white] (10,21) rectangle (22,32);
\fill[color=blue!50!white] (6,5) rectangle (7,12);
\fill[color=gray!50!white] (18,6) rectangle (22,17);
\fill[color=gray!50!white] (5,0) rectangle (7,4);
\fill[color=gray!50!white] (4,0) rectangle (5,3);
\fill[color=blue!50!white] (14,13) rectangle (15,16);
\fill[color=blue!50!white] (13,18) rectangle (20,21);
\fill[color=blue!50!white] (10,9) rectangle (11,20);
\fill[color=blue!50!white] (4,7) rectangle (6,9);
\fill[color=blue!50!white] (5,30) rectangle (7,32);
\fill[color=blue!50!white] (6,25) rectangle (8,30);
\fill[color=gray!50!white] (8,6) rectangle (18,7);
\fill[color=yellow] (5,29) rectangle (6,30);
\node () at (5.5,29.5) {$\rightarrow$};
\fill[color=blue!50!white] (7,11) rectangle (9,12);
\fill[color=blue!50!white] (8,8) rectangle (10,11);
\fill[color=blue!50!white] (4,9) rectangle (5,12);
\fill[color=blue!50!white] (9,22) rectangle (10,31);
\fill[color=yellow] (9,31) rectangle (10,32);
\node () at (9.5,31.5) {$\rightarrow$};
\fill[color=blue!50!white] (12,19) rectangle (13,21);
\fill[color=red] (3,9) rectangle (4,27);
\fill[color=yellow] (7,30) rectangle (8,31);
\node () at (7.5,30.5) {$\rightarrow$};
\fill[color=yellow] (5,9) rectangle (6,10);
\node () at (5.5,9.5) {$\rightarrow$};
\fill[color=blue!50!white] (15,16) rectangle (18,18);
\fill[color=blue!50!white] (11,14) rectangle (14,17);
\fill[color=blue!50!white] (5,10) rectangle (6,12);
\fill[color=yellow] (7,23) rectangle (8,24);
\node () at (7.5,23.5) {$\leftarrow$};
\fill[color=blue!50!white] (5,26) rectangle (6,29);
\fill[color=yellow] (4,26) rectangle (5,27);
\node () at (4.5,26.5) {$\leftarrow$};
\fill[color=yellow] (15,14) rectangle (16,15);
\node () at (15.5,14.5) {$\rightarrow$};
\fill[color=yellow] (3,7) rectangle (4,8);
\node () at (3.5,7.5) {$\leftarrow$};
\fill[color=red] (4,27) rectangle (5,32);
\fill[color=gray!50!white] (12,7) rectangle (18,11);
\fill[color=blue!50!white] (4,21) rectangle (8,23);
\fill[color=red] (3,29) rectangle (4,32);
\fill[color=blue!50!white] (12,11) rectangle (13,14);
\fill[color=gray!50!white] (21,17) rectangle (22,20);
\fill[color=yellow] (5,25) rectangle (6,26);
\node () at (5.5,25.5) {$\leftarrow$};
\fill[color=yellow] (10,20) rectangle (11,21);
\node () at (10.5,20.5) {$\leftarrow$};
\fill[color=yellow] (14,16) rectangle (15,17);
\node () at (14.5,16.5) {$\leftarrow$};
\fill[color=yellow] (4,6) rectangle (5,7);
\node () at (4.5,6.5) {$\leftarrow$};
\fill[color=blue!50!white] (11,10) rectangle (12,12);
\fill[color=blue!50!white] (11,17) rectangle (12,19);
\fill[color=gray!50!white] (16,11) rectangle (18,15);
\fill[color=yellow] (3,3) rectangle (4,4);
\node () at (3.5,3.5) {$\rightarrow$};
\fill[color=blue!50!white] (7,6) rectangle (8,10);
\fill[color=blue!50!white] (4,23) rectangle (5,26);
\fill[color=blue!50!white] (7,24) rectangle (8,25);
\fill[color=yellow] (3,28) rectangle (4,29);
\node () at (3.5,28.5) {$\rightarrow$};
\fill[color=gray!50!white] (17,15) rectangle (18,16);
\fill[color=gray!50!white] (3,0) rectangle (4,2);
\fill[color=yellow] (11,19) rectangle (12,20);
\node () at (11.5,19.5) {$\leftarrow$};
\fill[color=blue!50!white] (20,19) rectangle (21,21);
\fill[color=blue!50!white] (21,20) rectangle (22,21);
\fill[color=blue!50!white] (4,3) rectangle (5,6);
\fill[color=yellow] (3,2) rectangle (4,3);
\node () at (3.5,2.5) {$\leftarrow$};
\fill[color=yellow] (13,17) rectangle (14,18);
\node () at (13.5,17.5) {$\leftarrow$};
\fill[color=gray!50!white] (10,7) rectangle (12,9);
\fill[color=gray!50!white] (14,11) rectangle (16,13);
\fill[color=yellow] (6,24) rectangle (7,25);
\node () at (6.5,24.5) {$\leftarrow$};
\fill[color=yellow] (3,8) rectangle (4,9);
\node () at (3.5,8.5) {$\rightarrow$};
\fill[color=yellow] (5,5) rectangle (6,6);
\node () at (5.5,5.5) {$\leftarrow$};
\fill[color=yellow] (8,22) rectangle (9,23);
\node () at (8.5,22.5) {$\leftarrow$};
\fill[color=gray!50!white] (19,17) rectangle (21,18);
\fill[color=yellow] (7,10) rectangle (8,11);
\node () at (7.5,10.5) {$\rightarrow$};
\fill[color=gray!50!white] (9,7) rectangle (10,8);
\fill[color=gray!50!white] (11,9) rectangle (12,10);
\fill[color=blue!50!white] (5,6) rectangle (6,7);
\fill[color=blue!50!white] (18,17) rectangle (19,18);
\fill[color=yellow] (9,21) rectangle (10,22);
\node () at (9.5,21.5) {$\leftarrow$};
\fill[color=blue!50!white] (11,13) rectangle (12,14);
\fill[color=yellow] (15,15) rectangle (16,16);
\node () at (15.5,15.5) {$\leftarrow$};
\fill[color=gray!50!white] (13,11) rectangle (14,12);
\fill[color=blue!50!white] (7,31) rectangle (8,32);
\fill[color=yellow] (5,4) rectangle (6,5);
\node () at (5.5,4.5) {$\rightarrow$};
\fill[color=blue!50!white] (3,4) rectangle (4,7);
\fill[color=blue!50!white] (14,17) rectangle (15,18);
\fill[color=yellow] (9,11) rectangle (10,12);
\node () at (9.5,11.5) {$\rightarrow$};
\fill[color=blue!50!white] (5,23) rectangle (6,25);
\fill[color=blue!50!white] (12,17) rectangle (13,18);
\fill[color=blue!50!white] (16,15) rectangle (17,16);
\fill[color=blue!50!white] (13,12) rectangle (14,13);
\fill[color=yellow] (3,27) rectangle (4,28);
\node () at (3.5,27.5) {$\leftarrow$};
\fill[color=gray!50!white] (20,18) rectangle (21,19);
\fill[color=blue!50!white] (6,23) rectangle (7,24);
\fill[color=gray!50!white] (15,13) rectangle (16,14);
\fill[color=yellow] (11,12) rectangle (12,13);
\node () at (11.5,12.5) {$\rightarrow$};
\fill[color=blue!50!white] (8,21) rectangle (9,22);
\fill[color=gray!50!white] (6,4) rectangle (7,5);
\fill[color=blue!50!white] (11,20) rectangle (12,21);
\fill[color=blue!50!white] (8,7) rectangle (9,8);
\fill[color=yellow] (13,13) rectangle (14,14);
\node () at (13.5,13.5) {$\rightarrow$};
\draw (0,0) grid (22,32);

\draw[ultra thick] (3,2) -- (3,32);
\foreach \x in {3,4,5,...,20}{
	\draw[ultra thick] (\x,\x-1) -- +(1,0) -- +(1,1);
}
\draw[ultra thick] (21,20) -- +(1,0);

\draw[->] (1.5,33.25) -- (1.5,32.25);
\picnote{(1.5,33.25)}{W};
\draw[->] (4,33.25) -- (4,32.25);
\picnote{(4,33.25)}{R};
\draw[->] (7,33.25) -- (7,32.25);
\picnote{(7,33.25)}{B};
\draw[->] (9.5,33.25) -- (9.5,32.25);
\picnote{(9.5,33.25)}{Y};
\draw[->] (23.5,6.5) -- (22.25,6.5);
\picnote{(23.5,6.5)}{Gr};

\end{tikzpicture}
\end{center}
\caption{(A sofic projection of) the base of the computation cone. The limits of the cone are indicated by thick lines, and the zig-zag head by a yellow tile with an arrow indicating its direction. The red counter is incremented in the first two computation steps.}
\label{fig:CounterComputing}
\end{figure}

We claim that there are only countably many configurations in $X$. Namely, if one of the rows is in the orbit of $\phi((m_\ell, h, m_r), (n_1, \ldots, n_k))$ for finite $m_\ell$ and $m_r$, then the configuration contains a full computation cone, and thus corresponds to an infinite computation of $M$. If none of the rows are of this form, then the configuration contains at most one back-and-forth sweep of the zig zag head, and it is easy to see that there are only countably many such degenerate configurations.

Finally, south determinism is proved as follows if $M$ is reversible. First, the movement of the zig zag head is reversible and independent of $M$. It is a simple case analysis that the operation of updating a counter during the sweep is reversible. Consider then the step where the zig zag head is situated at the left border of the computation cone, and nondeterministically chooses a transition from $\delta$. The inverse step can be deterministically chosen depending on which counters contain the value $0$.
\end{construction}

In \cite{JeVa11}, a similar encoding of `computation in a cone' is used, but instead of counter machines, computation histories of Turing machines are embedded into configurations of countable SFTs. While both approaches have their merits, we feel that it is slightly more obvious how the counter machine construction works and why the resulting subshift is countable. Also, in our approach, the resulting subshift has the \PProperty.

Next, we study the computational power of the $k$th derivative of a countable two-dimensional SFT, which turns out to possibly climb very high in the arithmetical hierarchy. We start with an upper bound, which we then reach with a construction. A generalization of the following lemma was proved in \cite[Lemma 1.2 (3)]{CeClSmSoWa86}, but we include a proof for completeness.

\begin{lemma}
\label{lem:IncByTwo}
Given a two-dimensional $\PI^0_k$ subshift $X$ and a pattern $P$, it is $\PI^0_{k+2}$ whether $P \sqsubset X^{(1)}$.
\end{lemma}

\begin{proof}
Given $X$ and $P$, we have $P \sqsubset X^{(1)}$ iff $P \sqsubset X$ and for all $n \in \N$, there exist two distinct equal-sized extensions $Q_1, Q_2 \sqsubset X$ of $P$ that agree on the $(n \times n)$-square around $P$. This is clearly $\PI^0_{k+2}$. 
\end{proof}

The following construction shows that the bound given by Lemma~\ref{lem:IncByTwo} on the complexity of $k$th derivatives of $\PI^0_1$ subshifts is strict, and can be attained by a single deterministic countable SFT with the \PProperty. In particular, it implies that Proposition~\ref{prop:1DSoficNiceness} is far from true in dimension two, since two-dimensional sofic shifts are $\PI^0_1$, while their derivatives may be $\PI^0_3$-complete, and thus highly nonsofic. The rank of the subshift we build will be $\omega + k$ for some finite $k$. We start with a definition, and a classical recursion-theoretic lemma.

\begin{definition}
For $k \in \N$, denote by $\Phi_k$ the set of first-order arithmetical formulas with $k$ free variables and only bounded quantifiers. For $k, \ell \in \N$, denote by $\phi^k_\ell$ the $\ell$th formula in $\Phi_k$, ordered first by length and then lexicographically.
\end{definition}

\begin{lemma}[Lemma 2 in \cite{KrShWa60}]
\label{lemma:Infty}
Let $k \in \N$ and $\phi \in \Phi_{2k+1}$. Then there exists $\psi \in \Phi_{k+1}$, uniformly computable from $\phi$ and $k$, such that
\[ \forall n_1 : \exists n_2 : \cdots \forall n_{2k-1} : \exists n_{2k} : \forall n_{2k+1} : \phi(n_1, \ldots, n_{2k+1}) \]
is equivalent to
\[ \exists^\infty n_1 : \exists^\infty n_2 : \cdots \exists^\infty n_k : \forall n_{k+1} : \psi(n_1, \ldots, n_{k+1}). \]
\end{lemma}

We denote $\psi = I(\phi)$ in the above lemma. With this result, we can transform alternating quantifiers into infinitary ones, and the application to derivatives is rather straightforward.

\begin{theorem}
\label{thm:DerivativesCanCompute}
There exists a deterministic countable two-dimensional SFT $X$ with the \PProperty such that the problem whether $P \sqsubset X^{(k)}$ for a given pattern $P$ is $\PI^0_{2k+1}$-complete, for all $k \in \N$.
\end{theorem}

\begin{proof}
As an illustration of the idea, consider the closure of the subset of $\{0,1\}^\N$ consisting of points of the form
\[ 0^\ell 1 0^k 1 0^{n_1} 1 0^{n_2} \cdots 0^{n_k} 1 0 \INF, \]
where $I(\phi_l^{2k+1})(n_1, n_2, \ldots, n_k, n_{k+1})$ is true for all $n_{k+1}$. This set is $\PI^0_1$-complete. Clearly, the derivative of this closed set contains only those points of the form
\[ 0^\ell 1 0^k 1 0^{n_1} 1 0^{n_2} \cdots 0^{n_{k-1}} 1 0^\infty, \]
where $I(\phi_\ell^{2k+1})(n_1, n_2, \ldots, n_{k-1}, n_k, n_{k+1})$ holds for infinitely many $n_k$ and all $n_{k+1}$. Thus the derivative is $\PI^0_3$-complete, and we could verify by induction that the $n$th derivative is $\PI^0_{2n+1}$-complete. The construction of $X$ is an implementation of the same idea using a suitable counter machine and Construction~\ref{con:CounterMachine}.

\begin{algorithm}
\caption{The program of the counter machine $M$}\label{alg:CounterGuess}
\begin{algorithmic}[1]

\State $k \gets -1$
\Repeat
	\State $k \gets k + 1$
	\State \Choose{$b \in \{0, 1\}$}
\Until{$b = 1$}
\State $\ell \gets -1$
\Repeat
	\State $\ell \gets \ell + 1$
	\State \Choose{$b \in \{0, 1\}$}
\Until{$b = 1$}
\ForAll{$i \in \{1, \ldots, k\}$}
	\State $n_i \gets -1$
	\Repeat
		\State $n_i \gets n_i + 1$
		\State \Choose{$b \in \{0, 1\}$} \label{ln:ChooseB}
	\Until{$b = 1$}
\EndFor
\ForAll{$n_{k+1} \in \N$}
	\If{\textbf{not} $I(\phi^{2k+1}_\ell)(n_1, \ldots, n_{k+1})$}
		\State \textbf{reject}
	\EndIf
\EndFor

\end{algorithmic}
\end{algorithm}

Let $M$ be a two-input counter machine that operates as per Algorithm~\ref{alg:CounterGuess}. The machine $M$ simply guesses the parameters $k$ and $\ell$, then guesses the $k$ numbers $n_1, \ldots, n_k$, and finally checks that $I(\phi_\ell^{2k+1})(n_1, n_2, \ldots, n_{k-1}, n_k, n_{k+1})$ holds for all $n_{k+1} \in \N$. By Lemma~\ref{thm:Morita}, we may assume $M$ is reversible. Note that the set of infinite computations of $M$ is countable, since the choice $b = 1$ can be made only finitely many times during a computation. We plug $M$ in Construction~\ref{con:CounterMachine} to obtain the corresponding countable deterministic SFT $X$ with the \PProperty. Even after applying Lemma~\ref{thm:Morita}, the nondeterministic guesses of $M$ are visible in the SFT, so a finite initial part of a computation cone where the parameters $k$ and $\ell$ have been chosen occurs in $X^{(k)}$ iff
\[ \exists^\infty n_1 : \cdots \exists^\infty n_k : \forall n_{k+1} : I(\phi^{2k+1}_l)(n_1, \ldots, n_{k+1}) \]
is true. But by Lemma~\ref{lemma:Infty}, this is equivalent to
\[ \forall n_1 : \exists n_2 : \cdots \forall n_{2k-1} : \exists n_{2k} : \forall n_{2k+1} : \phi^{2k+1}_l(n_1, \ldots, n_{2k+1}), \]
and thus the subshift $X^{(k)}$ is $\PI^0_{2k+1}$-hard in the sense of the claim. Since it reaches the upper bound given by Lemma~\ref{lem:IncByTwo}, it is actually $\PI^0_{2k+1}$-complete.
\end{proof}

Next, we prove the result that deterministic countable SFTs with the \PProperty can have arbitrarily high computable Cantor-Bendixson ranks. We use both Lemma~\ref{thm:Morita} and Lemma~\ref{thm:Morita2} to make the construction deterministic, and to optimize the simulation overhead to $4$ (from the value $11$ in Theorem 4.3 of \cite{JeVa11}). Even without these optimization steps for the counter machine, we would obtain arbitrarily high computable Cantor-Bendixson ranks with countable SFTs with the \PProperty.

\begin{theorem}
\label{thm:BigRanks}
For any countable $\PI^0_1$ set $S \subset \{0, 1\}^\N$, there is a countable deterministic SFT $X$ with the \PProperty for which $\rank(X) = \rank(S) + 4$ holds.
\end{theorem}

\begin{proof}
Without loss of generality we assume that $S \neq \emptyset$. Since $S$ is $\PI^0_1$, there exists a Turing machine $T_S$ which outputs a potentially infinite list of words $F \subset \{0, 1\}^*$ such that
\[ S = \{ x \in \{0, 1\}^\N \;|\; \forall n \in \N : x_{[0, n-1]} \notin F \} . \]
We denote by $T_S(n)$ the (finite) list of words produced by $T_S$ after $n$ computation steps. We define a counter machine $M_S$ by Algorithm~\ref{alg:Cofinal}.

\begin{algorithm}
\caption{The program of the counter machine $M_S$}\label{alg:Cofinal}
\begin{algorithmic}[1]

\State $n \gets 0$
\State $W = \{ \lambda \}$
\Loop
	\State \Choose{$w_n \in \{0, 1\}$}
	\State $W \gets W \cup \{w_0 \cdots w_n\}$
	\If{$W \cap T_S(n) \neq \emptyset$}
		\State \textbf{reject}
	\EndIf
	\State $n \gets n + 1$
\EndLoop

\end{algorithmic}
\end{algorithm}

Using Lemma~\ref{thm:Morita} and Lemma~\ref{thm:Morita2}, we can assume that $M_S$ is reversible and uses only two counters. It is clear that the infinite computation histories of $M_S$ form a countable set, since each corresponds to an element of $S$. We then apply Construction~\ref{con:CounterMachine} to $M_S$ to obtain our SFT $X$.

For all $\vec n \in \Z^2$, the set $X_{\vec n}$ of configurations of $X$ where the computation cone is anchored at $\vec n$ is homeomorphic to $S$, since the choices of the $w_n$ are visible in the configurations and are the only source of nondeterminism in $M_S$. Furthermore, each set $X_{\vec n}$ is open in $X$, so by Lemma~\ref{lem:OpenDerivative}, we have
\[ X^{(\lambda)} = \bigcup_{\vec n \in \Z^2} X_{\vec n}^{(\lambda)} \cup Y, \]
for every ordinal $\lambda \leq \rank(S)$, where $Y$ contains only degenerate configurations. Furthermore, we show that the Cantor-Bendixson rank of $Y$ is exactly $4$. Let $y \in Y$ be arbitrary, so that $y$ does not contain the base of a computation cone, and let $\alpha = \rank(Y)$. If $y$ contains the right border of a computation cone, it cannot contain any counters, and its rank is then at most $\alpha - 2$. If it contains the left border, we claim it cannot contain both counters. This is because $M$ is executing an algorithm whose memory consumption increases with time (it remembers the set $W$ which increases in size), so for all $m \geq 0$ there exists $t \geq 0$ such that after $t$ computation steps, the sum of the counter values is always at least $m$. This means that the rank of $y$ is at most $\alpha - 4$. Finally, Figure~\ref{fig:Rank5} shows an example configuration of rank $\alpha - 4$, which contains the left border of a computation cone, one counter and the zig zag head. A finite counter value can be found, since the counter machines given by Lemma~\ref{thm:Morita2} decrement one of their counters to zero infinitely many times during all infinite computations. This shows that $\alpha = 4$, and the theorem is proved.
\end{proof}

\begin{figure}[ht]
\begin{center}
\begin{tikzpicture}[scale=0.35]

\fill[color=blue!50!white] (10,3) rectangle (19,14);
\fill[color=blue!50!white] (14,14) rectangle (19,16);
\fill[color=blue!50!white] (13,0) rectangle (19,3);
\fill[color=red] (5,8) rectangle (8,11);
\fill[color=red] (3,11) rectangle (5,16);
\fill[color=red] (3,0) rectangle (6,7);
\fill[color=blue!50!white] (12,14) rectangle (14,15);
\fill[color=blue!50!white] (12,1) rectangle (13,3);
\fill[color=red] (5,13) rectangle (8,16);
\fill[color=blue!50!white] (7,0) rectangle (9,4);
\fill[color=yellow] (5,11) rectangle (6,12);
\node () at (5.5,11.5) {$\rightarrow$};
\fill[color=blue!50!white] (8,5) rectangle (10,13);
\fill[color=blue!50!white] (8,14) rectangle (11,16);
\fill[color=yellow] (7,12) rectangle (8,13);
\node () at (7.5,12.5) {$\rightarrow$};
\fill[color=yellow] (9,13) rectangle (10,14);
\node () at (9.5,13.5) {$\rightarrow$};
\fill[color=blue!50!white] (9,0) rectangle (11,2);
\fill[color=yellow] (7,5) rectangle (8,6);
\node () at (7.5,5.5) {$\leftarrow$};
\fill[color=blue!50!white] (11,15) rectangle (13,16);
\fill[color=red] (6,0) rectangle (7,6);
\fill[color=red] (6,11) rectangle (8,12);
\fill[color=red] (7,6) rectangle (8,8);
\fill[color=yellow] (11,1) rectangle (12,2);
\node () at (11.5,1.5) {$\leftarrow$};
\fill[color=blue!50!white] (11,0) rectangle (12,1);
\fill[color=blue!50!white] (9,2) rectangle (10,3);
\fill[color=yellow] (11,14) rectangle (12,15);
\node () at (11.5,14.5) {$\rightarrow$};
\fill[color=yellow] (10,2) rectangle (11,3);
\node () at (10.5,2.5) {$\leftarrow$};
\fill[color=yellow] (13,15) rectangle (14,16);
\node () at (13.5,15.5) {$\rightarrow$};
\fill[color=red] (4,9) rectangle (5,11);
\fill[color=yellow] (4,8) rectangle (5,9);
\node () at (4.5,8.5) {$\leftarrow$};
\fill[color=red] (5,12) rectangle (7,13);
\fill[color=red] (3,7) rectangle (4,9);
\fill[color=red] (6,7) rectangle (7,8);
\fill[color=yellow] (8,4) rectangle (9,5);
\node () at (8.5,4.5) {$\leftarrow$};
\fill[color=yellow] (6,6) rectangle (7,7);
\node () at (6.5,6.5) {$\leftarrow$};
\fill[color=red] (4,7) rectangle (5,8);
\fill[color=yellow] (5,7) rectangle (6,8);
\node () at (5.5,7.5) {$\leftarrow$};
\fill[color=yellow] (12,0) rectangle (13,1);
\node () at (12.5,0.5) {$\leftarrow$};
\fill[color=blue!50!white] (9,4) rectangle (10,5);
\fill[color=blue!50!white] (11,2) rectangle (12,3);
\fill[color=yellow] (9,3) rectangle (10,4);
\node () at (9.5,3.5) {$\leftarrow$};
\fill[color=yellow] (3,10) rectangle (4,11);
\node () at (3.5,10.5) {$\rightarrow$};
\fill[color=blue!50!white] (7,4) rectangle (8,5);
\fill[color=blue!50!white] (8,13) rectangle (9,14);
\fill[color=yellow] (3,9) rectangle (4,10);
\node () at (3.5,9.5) {$\leftarrow$};
\draw (0,0) grid (19,16);

\draw[ultra thick] (3,0) -- (3,16);

\draw[->] (1.5,17.5) -- (1.5,16.25);
\picnote{(1.5,17.5)}{W};
\draw[->] (5.5,17.5) -- (5.5,16.25);
\picnote{(5.5,17.5)}{R};
\draw[->] (10.5,17.5) -- (10.5,16.25);
\picnote{(10.5,17.5)}{B};
\draw[->] (13.5,17.5) -- (13.5,16.25);
\picnote{(13.5,17.5)}{Y};

\end{tikzpicture}
\end{center}
\caption{A configuration of rank $\alpha - 4$ in $Y$, from the proof of Theorem~\ref{thm:BigRanks}. The computation area, and the other counter, are infinite to the right. A rank $\alpha - 3$ configuration would have, for example, an infinite value also in the red counter. A rank $\alpha - 2$ configuration would lose the zig zag head, and rank $\alpha - 1$ configurations are periodic.}
\label{fig:Rank5}
\end{figure}
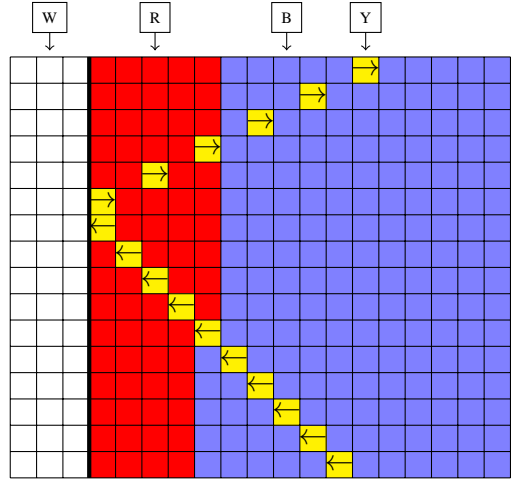

In \cite{CeClSmSoWa86} it is proved that the rank of a nonempty countable $\PI^0_1$ set can be any recursive successor ordinal. Thus, Theorem~\ref{thm:BigRanks} shows that for any recursive ordinal $\alpha$ \emph{not} of the form $\beta + n$ for a limit ordinal $\beta$ and $n \in \{0,1,2,3,4\}$, there exists a countable deterministic two-dimensional SFT with Cantor-Bendixson rank exactly $\alpha$. In \cite{BaDuJe08}, it it shown that the ranks $\beta$ and $\beta + 1$ cannot be achieved for any countable subshift, leaving only the cases $\beta + 2$, $\beta + 3$ and $\beta + 4$ unaswered.

With the same construction, we also obtain the upper bound $6$ for the smallest possible rank of a countable SFT with uncomputable points. For $\PI^0_1$ subshifts (and closed sets in general), the smallest such rank is known to be $2$ \cite{CeDaToWy12}. Applied to such a set, the construction in Theorem~\ref{thm:BigRanks} gives a countable SFT with rank $6$, and clearly preserves computability and uncomputability of non-degenerate points.

Finally, we give our geometric construction of an infinite Cantor-Bendixson rank as another, perhaps more natural, example of how infinite ranks might arise in countable SFTs. The construction has the \PProperty, but is not deterministic.

\begin{example}
\label{ex:Diamonds}
We give an example of a countable two-dimensional SFT $X$ of rank at least $\omega$ with the \PProperty. Consider the one-dimensional subshift containing points of the form
\[ {}^\infty 0 a^k 0^{m_1} a^{k-1}b 0^{m_2} a^{k-2}b^2 0^{m_3} \cdots 0^{m_k} b^k 0^\infty, \]
where $k \in \N$ and $m_i \in \N$ for all $i$ are arbitrary. For all $k$, the subshift contains configurations with $k$ `islands' floating in a sea of $0$'s, but no configuration contains an infinite number of islands. This is a countable subshift with infinite rank, and in the following, we construct a two-dimensional SFT $X$ that uses exactly the same idea.

The SFT $X$ contains one infinite horizontal \emph{dedicated line}. The top and bottom halves are colored differently. On the line one may have (perhaps infinite) \emph{diamonds}, colored red and blue, whose left and right corners must be on the dedicated line. The diamonds must be nested, that is, a blue diamond must either contain a red diamond or be contained in one (not both) and vice versa. This is established by sending signals along the dedicated line. Two distinct diamonds may not overlap, unless one is completely inside the other (including a complete overlap). The insides of the diamonds are colored differently from their outsides.

From the top (bottom) corner of every red (blue) diamond, a \emph{decrement signal} is sent to the right (left, respectively). Also, the top (bottom) corner of every red (blue) diamond must absorb one decrement signal traveling one tile above (below) it. The area between the line and a signal is colored differently from its complement. See Figure~\ref{fig:Diamonds} for a clarifying picture.

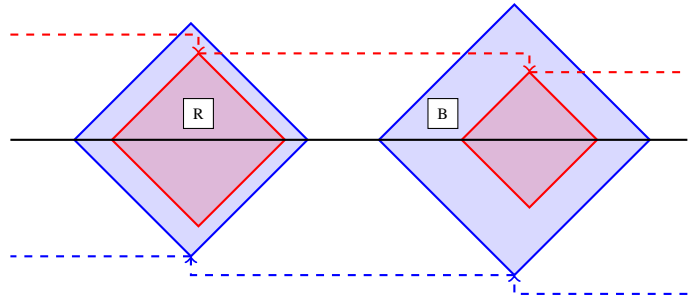
\begin{figure}[ht]

\begin{center}
\begin{tikzpicture}

\diamondlord{2.4}{0}{1.55}{blue};
\diamondlord{2.5}{0}{1.15}{red};
\diamondlord{6.7}{0}{1.8}{blue};
\diamondlord{6.9}{0}{.9}{red};

\draw[blue,thick,dashed,->] (9,-2.05) -- (6.7,-2.05) -- (6.7,-1.8);
\draw[blue,thick,dashed,->] (6.7,-1.8) -- (2.4,-1.8) -- (2.4,-1.55);
\draw[blue,thick,dashed] (2.4,-1.55) -- (0,-1.55);

\draw[red,thick,dashed,->] (0,1.4) -- (2.5,1.4) -- (2.5,1.15);
\draw[red,thick,dashed,->] (2.5,1.15) -- (6.9,1.15) -- (6.9,.9);
\draw[red,thick,dashed] (6.9,.9) -- (9,.9);

\draw[thick] (0,0) -- (9,0);

\picnote{(2.5,0.35)}{R};
\picnote{(5.75,0.35)}{B};

\end{tikzpicture}
\end{center}

\caption{The diamond construction.}
\label{fig:Diamonds}
\end{figure}

We first show that $X$ is countable. Indeed, for each $(n,m) \in \N^2$, if a configuration $x$ of $X$ contains nested diamonds of sizes $n$ and $m$, then there are at most $n+m-1$ pairs of diamonds in $x$, since the size of the red (blue) diamonds decreases to the right (left). The number of ways to arrange these points and the surrounding diamonds is countable. One can also check that the number of exceptional points (ones containing, say, an infinite diamond or just signals) is countable.

Next, we show that $X^{(\omega)}$ is a nonempty set of finite rank. First, the isolated points of $X$ are exactly those that contain finite red and blue diamonds, and whose rightmost red and leftmost blue diamonds are of size $1$. In general, if $x \in X$ contains finite red and blue diamonds, we say that $x$ \emph{has type $(n,m)$} if the rightmost red diamond (leftmost blue diamond) is of size $n$ ($m$). It is then easy to see that for all $k \in \N$, the set $X^{(k)}$ will contain all of $X$, except for the points of type $(n,m)$ with $n + m < k + 2$. Then $X^{(\omega)}$ is nonempty, but will consist of only exceptional points, and clearly $X^{(\omega+k)} = \emptyset$ for some finite $k$. 
\end{example}

\section{Chains}

In this section, we investigate whether the subpattern poset of a countable SFT necessarily has the ascending or descending chain condition. It turns out that it always has the ACC, but need not have the DCC, that is, infinite upward chains are impossible, but infinite downward chains are possible. A proof of the following theorem can be found in \cite[Theorem 3.7]{BaDuJe08}. We also give a less sophisticated, more hands-on proof in \cite{SaTo12b}.

\begin{theorem}
The subpattern poset of a countable subshift has the ACC.
\end{theorem}

Arbitrary widths for the subpattern poset, that is, infinite antichains, are possible and in fact are harder to avoid than produce. Example~\ref{ex:Grid} is one example, but even the one-dimensional countable SFT $\B^{-1}(0^*1^*2^*)$ has an infinite antichain.

\subsection{An infinite downward chain}

While upward chains are easy to show impossible and antichains are trivial to find, the case of a downward chain is more interesting, and was left open in \cite{BaDuJe08}. This is the content of the following theorem:

\begin{theorem}
\label{thm:Chain}
There exists a deterministic countable two-dimensional SFT $X$ such that an infinite downward chain can be order-embedded in $\SP(X)$.
\end{theorem}

\begin{proof}
We illustrate the idea of the construction with the following one-dimensional `subshift' on the countably infinite alphabet $\N$, which is generated by a single configuration $\tilde x \in \N^\Z$ of the Baire space. First, define $\tilde x^1 \in \{0,1\}^\Z$ by $\tilde x^1_{2^i} = 1$ for all $i \in \N$, and $\tilde x^1_i = 0$ everywhere else. We inductively define $\tilde x^{n+1} \in \{0, \ldots, n+1\}^\Z$ such that $\tilde x^{n+1}_i \neq x^n_i$ only if $\tilde x^{n+1}_i = n+1$ and $\tilde x^n_i = 0$, and if $(n+1)0^k(n+1)$ occurs in $\tilde x^{n+1}$, then $k = 2^i-1$ for some $i \in \N$. To define $\tilde x^{n+1}$, we go through all segments $\tilde x^n_{[a, a + 2^i]} = n0^{2^i-1}n$, and rewrite the positions $a + 2^{i-1} + 2^j$ for $j \in \{0, \ldots, i-2\}$ with the letter $n+1$. Then, $\tilde x$ is defined as the limit of the sequence $(\tilde x^n)_{n \in \N}$. So,
\[ \tilde x = \cdots 00.1101002100002201000000002202000100000000000000002202003200000001000 \cdots \]
The `subshift' $\tilde X = \overline{\mathcal{O}(\tilde x)}$ contains the infinite decreasing chain $(\tilde y^n)_{n \in \N}$, where each $\tilde y^n$ is obtained from $\tilde x$ by incrementing each nonzero coordinate by $n$ (some prefixes of $\tilde y^1$ are already visible in $\tilde x$ above). Note that $\tilde X$ is a countable set, but not compact as a topological space.

Our construction is basically the implementation of $\tilde X$ as a deterministic countable two-dimensional SFT, and we build it in several intermediate steps. The cells of our SFTs are colored with either \emph{main colors} (such as white or blue) or \emph{border colors} that are represented by black tiles in the figures. The border colors serve mostly to separate the different colors from each other, but may also contain a bounded amount of data (a truth value or a finite counter, for example). See Table~\ref{tab:MainColors} for the labels of the colors used in our figures. A configuration consists of large areas of different main colors separated by the border colors, which we describe in geometric terms. Each of the areas resembles a geometric shape (a triangle or a half-plane), and the borders are discrete versions of rational lines.

Let $x$ be a configuration of some of our SFTs, and let $c$ be a main color. We say that two coordinates $\vec n, \vec m \in \Z^2$ are \emph{$c$-adjacent} in $x$, if $\| \vec n - \vec m \| \leq 1$ and $x_{\vec n}$ and $x_{\vec m}$ both have the color $c$. They are \emph{$c$-connected} if there exists a chain $(\vec{n}_i)_{i=0}^k$ of coordinates such that $\vec{n}_0 = \vec n$, $\vec{n}_k = \vec m$ and each pair $\vec{n}_i, \vec{n}_{i+1}$ is $c$-adjacent. A set $D \subset \Z^2$ is \emph{$c$-connected} in $x$ if $x_{\vec n} = c$ for all $\vec n \in D$, and all coordinate pairs $\vec n, \vec m \in D$ are $c$-connected. A maximal $c$-connected set $D$ in $x$ is called a $c$ shape, where `shape' is the geometric figure that $D$ resembles. For example, a triangular area of red cells would be called a red triangle.

We start by constructing a countable two-dimensional SFT $X_1$ corresponding to $\tilde x^1$. It is defined by the patterns of size $3 \times 3$ occurring in Figure~\ref{fig:ChainPicSimple}. To see the correspondence explicitly, define a continuous map $\phi_1 : X_1 \to \{0,1\}^\Z$ by $\phi_1(x)_i = 1$ iff there exists $j \in \Z$ such that $(i,j)$ is the southeast corner of a red triangle in $x$. Then, $\overline{\mathcal{O}(\tilde x^1)} \subset \phi_1(X_1)$, and the only configurations in $\phi_1(X_1) - \overline{\mathcal{O}(\tilde x^1)}$ are those where a single $1$ lies on the left of the infinite pattern of $1$'s of $\tilde x^1$. They are generated by an infinite pattern of triangles as seen in Figure~\ref{fig:ChainPicSimple}, plus a single infinite red triangle in the southwest.

\begin{figure}[ht]
\begin{center}
\begin{tikzpicture}[scale=0.1]

\draw[color=gray] (0,0) grid (74,67);

\fill[color=blue!50] (51,0) rectangle (74,49);
\fill[color=blue!50] (12,0) rectangle (51,10);
\fill[color=red] (24,34) rectangle (34,43);
\fill[color=red] (27,43) rectangle (34,49);
\fill[color=blue!50] (37,10) rectangle (51,35);
\fill[color=black] (35,65) rectangle (68,66);
\fill[color=blue!50] (61,49) rectangle (74,59);
\fill[color=blue!50] (27,10) rectangle (37,25);
\fill[color=blue!50] (68,59) rectangle (74,66);
\fill[color=blue!50] (14,10) rectangle (27,12);
\fill[color=black] (34,33) rectangle (35,65);
\fill[color=red] (30,49) rectangle (34,55);
\fill[color=red] (13,18) rectangle (18,21);
\fill[color=black] (4,3) rectangle (5,5);
\fill[color=black] (17,15) rectangle (18,16);
\fill[color=black] (18,17) rectangle (19,33);
\fill[color=black] (29,27) rectangle (30,28);
\fill[color=black] (10,9) rectangle (11,17);
\fill[color=blue!50] (32,25) rectangle (37,30);
\fill[color=black] (23,41) rectangle (24,43);
\fill[color=blue!50] (34,30) rectangle (37,32);
\fill[color=black] (55,53) rectangle (56,54);
\fill[color=blue!50] (8,0) rectangle (12,6);
\fill[color=blue!50] (69,66) rectangle (74,67);
\fill[color=blue!50] (56,49) rectangle (61,54);
\fill[color=blue!50] (6,0) rectangle (8,4);
\fill[color=black] (11,17) rectangle (12,19);
\fill[color=blue!50] (22,12) rectangle (27,20);
\fill[color=red] (36,66) rectangle (66,67);
\fill[color=blue!50] (58,54) rectangle (61,56);
\fill[color=blue!50] (17,12) rectangle (22,15);
\fill[color=red] (21,34) rectangle (24,37);
\fill[color=black] (3,2) rectangle (4,3);
\fill[color=red] (22,37) rectangle (24,39);
\fill[color=black] (19,33) rectangle (34,34);
\fill[color=red] (25,43) rectangle (27,45);
\fill[color=blue!50] (15,12) rectangle (17,13);
\fill[color=blue!50] (43,35) rectangle (51,41);
\fill[color=black] (5,5) rectangle (7,7);
\fill[color=black] (28,26) rectangle (29,27);
\fill[color=red] (29,49) rectangle (30,53);
\fill[color=red] (31,55) rectangle (34,57);
\fill[color=black] (52,50) rectangle (53,51);
\fill[color=red] (17,21) rectangle (18,29);
\fill[color=blue!50] (66,59) rectangle (68,64);
\fill[color=blue!50] (63,59) rectangle (66,61);
\fill[color=red] (26,45) rectangle (27,47);
\fill[color=black] (21,37) rectangle (22,39);
\fill[color=black] (22,20) rectangle (23,21);
\fill[color=red] (9,10) rectangle (10,13);
\fill[color=blue!50] (20,15) rectangle (22,18);
\fill[color=red] (15,21) rectangle (17,25);
\fill[color=blue!50] (13,10) rectangle (14,11);
\fill[color=black] (42,40) rectangle (43,41);
\fill[color=black] (57,55) rectangle (58,56);
\fill[color=black] (24,43) rectangle (25,45);
\fill[color=black] (27,49) rectangle (28,51);
\fill[color=black] (9,7) rectangle (10,8);
\fill[color=black] (30,55) rectangle (31,57);
\fill[color=black] (11,9) rectangle (12,10);
\fill[color=red] (16,25) rectangle (17,27);
\fill[color=black] (14,23) rectangle (15,25);
\fill[color=black] (40,38) rectangle (41,39);
\fill[color=black] (18,16) rectangle (19,17);
\fill[color=black] (22,39) rectangle (23,41);
\fill[color=black] (23,21) rectangle (24,22);
\fill[color=blue!50] (24,20) rectangle (27,22);
\fill[color=blue!50] (46,41) rectangle (51,44);
\fill[color=black] (8,11) rectangle (9,13);
\fill[color=black] (12,17) rectangle (18,18);
\fill[color=blue!50] (48,44) rectangle (51,46);
\fill[color=blue!50] (36,32) rectangle (37,34);
\fill[color=blue!50] (42,35) rectangle (43,40);
\fill[color=blue!50] (33,30) rectangle (34,31);
\fill[color=blue!50] (53,49) rectangle (56,51);
\fill[color=blue!50] (4,0) rectangle (6,2);
\fill[color=blue!50] (40,35) rectangle (42,38);
\fill[color=blue!50] (35,32) rectangle (36,33);
\fill[color=black] (63,61) rectangle (64,62);
\fill[color=black] (6,4) rectangle (7,5);
\fill[color=black] (7,9) rectangle (8,11);
\fill[color=red] (20,34) rectangle (21,35);
\fill[color=red] (33,57) rectangle (34,61);
\fill[color=blue!50] (30,25) rectangle (32,28);
\fill[color=blue!50] (26,22) rectangle (27,24);
\fill[color=black] (31,57) rectangle (32,59);
\fill[color=blue!50] (47,44) rectangle (48,45);
\fill[color=blue!50] (18,15) rectangle (20,16);
\fill[color=black] (33,61) rectangle (34,63);
\fill[color=black] (54,52) rectangle (55,53);
\fill[color=black] (41,39) rectangle (42,40);
\fill[color=black] (35,66) rectangle (36,67);
\fill[color=blue!50] (50,46) rectangle (51,48);
\fill[color=blue!50] (57,54) rectangle (58,55);
\fill[color=black] (21,19) rectangle (22,20);
\fill[color=blue!50] (65,61) rectangle (66,63);
\fill[color=blue!50] (45,41) rectangle (46,43);
\fill[color=black] (16,27) rectangle (17,29);
\fill[color=black] (29,53) rectangle (30,55);
\fill[color=red] (28,49) rectangle (29,51);
\fill[color=black] (32,59) rectangle (33,61);
\fill[color=black] (25,45) rectangle (26,47);
\fill[color=blue!50] (59,56) rectangle (61,57);
\fill[color=black] (16,14) rectangle (17,15);
\fill[color=blue!50] (10,6) rectangle (12,8);
\fill[color=black] (6,7) rectangle (7,9);
\fill[color=black] (47,45) rectangle (48,46);
\fill[color=black] (64,62) rectangle (65,63);
\fill[color=black] (25,23) rectangle (26,24);
\fill[color=black] (48,46) rectangle (49,47);
\fill[color=blue!50] (29,25) rectangle (30,27);
\fill[color=blue!50] (39,35) rectangle (40,37);
\fill[color=blue!50] (19,16) rectangle (20,17);
\fill[color=blue!50] (52,49) rectangle (53,50);
\fill[color=black] (59,57) rectangle (60,58);
\fill[color=blue!50] (7,4) rectangle (8,5);
\fill[color=red] (14,21) rectangle (15,23);
\fill[color=black] (36,34) rectangle (37,35);
\fill[color=red] (23,39) rectangle (24,41);
\fill[color=blue!50] (28,25) rectangle (29,26);
\fill[color=blue!50] (38,35) rectangle (39,36);
\fill[color=blue!50] (67,64) rectangle (68,65);
\fill[color=blue!50] (23,20) rectangle (24,21);
\fill[color=black] (20,35) rectangle (21,37);
\fill[color=black] (58,56) rectangle (59,57);
\fill[color=black] (26,47) rectangle (27,49);
\fill[color=blue!50] (16,13) rectangle (17,14);
\fill[color=blue!50] (64,61) rectangle (65,62);
\fill[color=black] (2,0) rectangle (3,1);
\fill[color=black] (61,59) rectangle (62,60);
\fill[color=black] (49,47) rectangle (50,48);
\fill[color=black] (10,8) rectangle (11,9);
\fill[color=red] (32,57) rectangle (33,59);
\fill[color=blue!50] (5,2) rectangle (6,3);
\fill[color=black] (26,24) rectangle (27,25);
\fill[color=black] (28,51) rectangle (29,53);
\fill[color=black] (45,43) rectangle (46,44);
\fill[color=black] (32,30) rectangle (33,31);
\fill[color=black] (17,29) rectangle (18,31);
\fill[color=black] (31,29) rectangle (32,30);
\fill[color=black] (39,37) rectangle (40,38);
\fill[color=black] (20,18) rectangle (21,19);
\fill[color=black] (9,13) rectangle (10,15);
\fill[color=black] (4,2) rectangle (5,3);
\fill[color=black] (53,51) rectangle (54,52);
\fill[color=black] (15,25) rectangle (16,27);
\fill[color=black] (46,44) rectangle (47,45);
\fill[color=blue!50] (41,38) rectangle (42,39);
\fill[color=black] (7,5) rectangle (8,6);
\fill[color=black] (51,49) rectangle (52,50);
\fill[color=black] (3,1) rectangle (4,2);
\fill[color=blue!50] (3,0) rectangle (4,1);
\fill[color=black] (24,22) rectangle (25,23);
\fill[color=red] (8,10) rectangle (9,11);
\fill[color=black] (14,12) rectangle (15,13);
\fill[color=blue!50] (44,41) rectangle (45,42);
\fill[color=black] (12,19) rectangle (13,21);
\fill[color=black] (5,3) rectangle (6,4);
\fill[color=black] (34,32) rectangle (35,33);
\fill[color=black] (44,42) rectangle (45,43);
\fill[color=blue!50] (60,57) rectangle (61,58);
\fill[color=black] (66,66) rectangle (67,67);
\fill[color=black] (33,31) rectangle (34,32);
\fill[color=black] (8,9) rectangle (10,10);
\fill[color=blue!50] (55,51) rectangle (56,53);
\fill[color=black] (30,28) rectangle (31,29);
\fill[color=black] (27,25) rectangle (28,26);
\fill[color=blue!50] (54,51) rectangle (55,52);
\fill[color=black] (13,21) rectangle (14,23);
\fill[color=black] (8,6) rectangle (9,7);
\fill[color=black] (19,17) rectangle (20,18);
\fill[color=blue!50] (31,28) rectangle (32,29);
\fill[color=blue!50] (11,8) rectangle (12,9);
\fill[color=black] (15,13) rectangle (16,14);
\fill[color=black] (43,41) rectangle (44,42);
\fill[color=black] (65,63) rectangle (66,64);
\fill[color=black] (66,64) rectangle (67,65);
\fill[color=black] (35,33) rectangle (36,34);
\fill[color=blue!50] (49,46) rectangle (50,47);
\fill[color=black] (50,48) rectangle (51,49);
\fill[color=blue!50] (62,59) rectangle (63,60);
\fill[color=blue!50] (25,22) rectangle (26,23);
\fill[color=black] (13,11) rectangle (14,12);
\fill[color=black] (37,35) rectangle (38,36);
\fill[color=black] (12,10) rectangle (13,11);
\fill[color=blue!50] (21,18) rectangle (22,19);
\fill[color=blue!50] (9,6) rectangle (10,7);
\fill[color=black] (56,54) rectangle (57,55);
\fill[color=black] (38,36) rectangle (39,37);
\fill[color=red] (12,18) rectangle (13,19);
\fill[color=black] (62,60) rectangle (63,61);
\fill[color=black] (19,34) rectangle (20,35);
\fill[color=black] (68,66) rectangle (69,67);
\fill[color=black] (60,58) rectangle (61,59);

\picnote{(30,40)}{R};
\picnote{(10,40)}{W};
\picnote{(50,40)}{LB};

\end{tikzpicture}
\end{center}
\caption{A configuration corresponding to $\tilde x^1$.}
\label{fig:ChainPicSimple}
\end{figure}
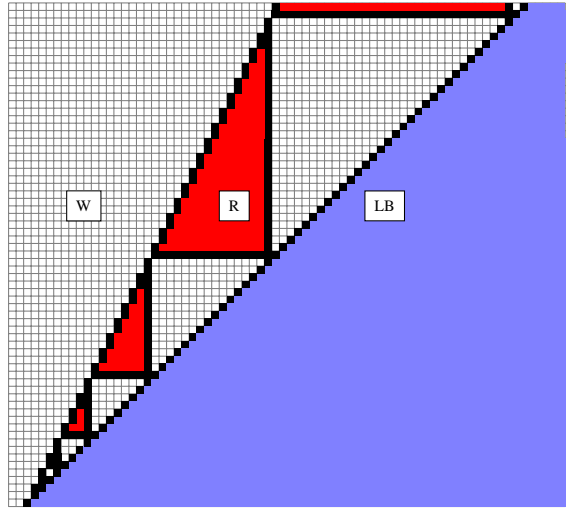

Define the map $\Delta : \N^\Z \to \{0,1\}^\Z$ by $\Delta(\tilde z)_i = 0$ iff $\tilde z_i = 0$. Next, we wish to modify $X_1$ and construct another SFT $X_2$ corresponding to $\Delta(\tilde x)$ via an analogous projection. To accomplish this, we split each red triangle into dark red and light red halves with a signal emitted by the northeast corners in the direction $(-1,-8)$. This signal intersects the south border of the triangle, and there we place the northeast corner of a smaller triangle. A partial configuration is shown in Figure~\ref{fig:ChainPicSimple2}, and the $3 \times 9$ patterns visible in the figure (apart from the blue vertical lines) define the SFT. From now on, the color red refers to both dark red and light red, so a red triangle is the union of a dark red triangle and the thinner light red triangle on its right. If we define $\phi_2 : X_2 \to \{0,1\}^\Z$ analogously to $\phi_1$, then $\overline{\mathcal{O}(\Delta(\tilde x))} \subset \phi_2(X_2)$ holds\footnote{Note that the closure $\overline{\mathcal{O}(\Delta(\tilde x))}$ is uncountable even though $\tilde{X} = \overline{\mathcal{O}(\tilde x)}$ is not.}, and $\phi_2(X_2) - \overline{\mathcal{O}(\Delta(\tilde x))}$ contains only anomalous configurations similar to those in $\phi_1(X_1) - \overline{\mathcal{O}(\tilde x^1)}$.

\input{ChainPicSimple2}

The top of every red triangle $T'$ in $X_2$ is attached either to the southwest corner or the south border of another red triangle $T$. In the first case, $T$ is the \emph{predecessor} of $T'$, and in the second, its \emph{parent}. Also, $T'$ is the \emph{successor} (\emph{child}, respectively) of $T$. We choose the SFT rules of $X_2$ so that every red triangle has either a predecessor or a parent. If the width of $T'$ is $w$, then the width of its predecessor is $2w$, and the width of its parent is $8w$. This is because the southeast corner of every red triangle meets the same light blue half plane. We also note here that $X_2$ is not countable, as one can construct an uncountable set of configurations as follows: start with a single finite red triangle $T_0$, and for all $n \in \N$, place the red triangle $T_{n+1}$, which can be chosen either as the predecessor or parent of $T_n$.

The construction of the first layer is almost finished, but we still need to make a small modification. Namely, define the SFT $X_3$ that has the same rules as $X_2$, except that a red triangle need not have a child nor a successor, even if its size would allow it. The SFT $X_3$ is again uncountable, for the same reason as $X_2$. The correspondence between $X_3$ and $\overline{\mathcal{O}(\Delta(\tilde x))}$ is not as strong as for $X_2$: if we define $\phi_3 : X_3 \to \{0,1\}^\Z$ analogously to $\phi_1$ and $\phi_2$, then every configuration of $\phi_3(X_3)$ is obtained from one of $\overline{\mathcal{O}(\Delta(\tilde x))}$ by changing certain $1$s to $0$s and possibly adding the anomalous lone $1$, but it is not easy to describe which changes can be made. However, the subshifts are still similar in spirit, and the complete SFT $X$ will be countable and contain an infinite downward chain for essentially the same reasons as $\tilde X$.

The next step in our construction is to attach to each red triangle a natural number, called its \emph{level}, analogously to the numbers of $\tilde x$. The level of a triangle should be the same as the level of its predecessor, or one greater than the level of its parent. This will force every chain of red triangles $(T_n)_{n \in \N}$, where each $T_{n+1}$ is either the parent or predecessor of $T_n$, to have only finitely many parent relations. Since we only have a finite alphabet in our use, we encode the level as the height of another geometric shape attached to each triangle, and use a set of signals to propagate them correctly. The new signals also have the secondary function of making the whole subshift southwest deterministic, since they control the creation of successors and children of the red triangles. Recall that a red triangle of $X_3$ need not have a successor nor a child; after the construction of the second layer is complete, their existence is determined by the width and level of the triangle.

We proceed with the definition of the second layer, $X_4$, which will be more technical than the first. First, we mark the east and south borders of each red triangle with a special border color, called the \emph{frame}. The first and second layer are independent, except that the frame, the light blue half plane and the southwest corners of light red triangles are visible in the second layer, and the second layer overlays some colors on the frame, which determine the formation of children and successors in the first layer. As with the first layer, the second layer will consist of geometric shapes of different colors, separated by border lines. The main colors are background (shown as white), yellow, dark yellow, orange, dark orange and blue. The only color of the second layer that can appear on top of the light blue half plane of the first layer is background. The configurations of $X_4$ are natural generalizations of the configuration in Figure~\ref{fig:Chain}, in the sense that the allowed corners and slopes between the uniformly colored areas are those visible in the figure (although not all possible interactions between discrete lines are shown). There is some freedom in the choice of the generalization, but the relevant implementation details are mentioned below. This SFT is glued to the first layer along the frame.

\input{ChainPicLevel2}

We now describe the different patterns seen in Figure~\ref{fig:Chain} and explain how the second layer works. The yellow-blue, dark yellow, orange and dark orange stripe extending to the northwest from the frame in each configuration of $X_4$ is called the \emph{level stripe}. The signals emitted in the direction $(-2,1)$ by the southwest and southeast corners of each red triangle (points $A$ and $B$ in Figure~\ref{fig:Chain}, respectively) are called \emph{guide signals}, since they guide the north and west borders of the level stripe. Each red triangle $T$ is assigned a number $L(T) \in \N$, called its \emph{level}, which is just the thickness of the level stripe attached to its south border (equivalently, half the thickness of the vertical level stripe attached to its east border), not counting one of the borders of the stripe. In Figure~\ref{fig:Chain}, the level of the triangle whose southeast corner is $B$ is $8$. Levels of triangles are analogous to the numbers in $\tilde x$. We do not allow levels below $8$ (for technical reasons that will be clear later), which can be enforced in $X_4$ by SFT rules. For convenience, we also define $W(T)$ as the width of $T$. The signals emitted in the direction $(-2,1)$ by the southwest corners of light red triangles (for example, point $C$ in Figure~\ref{fig:Chain}) are called \emph{branch signals}. Note that the branch signal \emph{moves two steps to the west} and turns to the direction $(0,-1)$ when it hits the north border of the level stripe (at point $D$).

Let $T$ be the red triangle whose southeast corner is $B$. If $T$ had no child, the south branch signal would be destroyed by the frame, as happens with the smaller triangles in Figure~\ref{fig:Chain}. Otherwise, it continues south and becomes the west border of the level stripe of the child $T'$ of $T$. As we mentioned above, as the northwest branch signal hits the north border of the level stripe of $T$ and turns to the south, it is also shifted two steps to the west. Thus, we have $L(T') = L(T) + 1$, or in other words, the level of a red triangle is one greater than the level of its parent (if one exists), or equal to the level of its predecessor. Since the set of possible levels has a lower bound, in particular in each chain $(T_i)_{i \in \N}$ of red triangles such that $T_{i+1}$ is either the predecessor or parent of $T_i$, there is only a finite number of parent relations. If the downward branch signal does not hit the south border of $T$ (so that the southwest corner of $T$ is on the border of an orange area), the southwest corner does not create a yellow area, but its guide signal instead destroys the level stripe of $T$. We call this the \emph{elimination rule}, and Figure~\ref{fig:Chain} contains an example of it at point $E$. The elimination rule is needed only to ensure the countability of the final SFT.

From the top of every red triangle $T$, we shoot a \emph{decision signal} to the southwest direction $(-1,-1)$, which then bounces from the west border of the vertical level stripe of $T$ to the direction $(1,-2)$, then southwest again, and so on. This is the border of the yellow and blue areas in Figure~\ref{fig:Chain}. Its purpose is to calculate, based on the level and size of $T$, whether it should have a child and/or a successor. In $X_4$, the frame contains a finite counter. At the top of every red triangle $T$, the counter of the southward frame segment is initialized to $0$. The counter is then incremented by $1$ every time the decision signal hits it, up to the maximum of $2$, and this number is transmitted to the south border of $T$. For example, in Figure~\ref{fig:Chain}, the counter is set to $0$ at point $F$ and is incremented twice on its way toward point $B$, so the value in the segment $AB$ is $2$. A successor (child) for $T$ is created if and only if the number is at least $1$ (exactly $2$, respectively). The number in the south border is called the \emph{determinant} of $T$, since we use it to determine whether $T$ has a successor or a child. Since the height of $T$ is $2W(T)$, a direct calculation shows that
\begin{enumerate}
\item if $W(T) < 3L(T)$, then the determinant of $T$ is $0$,
\label{item:One}
\item if $3L(T) \leq W(T) < 6L(T)$, then the determinant of $T$ is $1$, and
\label{item:Two}
\item if $6L(T) \leq W(T)$, then the determinant of $T$ is $2$.
\label{item:Three}
\end{enumerate}
Then, $T$ has no child or successor, only a successor, or both a child and a successor, depending on the respective value of its determinant. Now, the SFT $X$ combined from the two layers $X_3$ and $X_4$ is southwest deterministic. We are now finished with the construction of our SFT.

Next, we show that the combined SFT $X$ contains interesting infinite configurations. Let $(a,b) \in \Z^2$, and let $\ell, k \in \N$ be such that $\ell \geq 8$ and
\begin{equation}
\label{eq:IndHyp}
\frac{5}{3} 2^k \geq \ell \geq 8.
\end{equation}
Define $P(a,b,k,\ell)$ as the left infinite pattern over the alphabet of $X$ whose domain equals $S(a,b,k) = (-\infty, a] \times [b, b+2^{k+1}-1]$ containing a red triangle $T$ of width $w = 2^k$ and level $\ell$ at its east border, and white everywhere else (the southeast corner of $T$ is at $(a,b)$, which corresponds to point $B$ in Figure~\ref{fig:Chain}). We have denoted by $\alpha$ the line segment from $(a,b)$ to $(a,b+2^{k+1}-1)$ in Figure~\ref{fig:InductionPic}. Equation~\eqref{eq:IndHyp} guarantees that the horizontal level stripe of $T$ fits inside the pattern $P(a,b,k,\ell)$, and the choice for the constant $\frac{5}{3}$ is somewhat arbitrary. Now, the pattern $P(a,b,k,\ell)$ is locally valid, in the sense that it contains no forbidden pattern of $X$, and if it can be extended to the quarter plane $Q(a,b,k) = (-\infty, a] \times (-\infty, b+2^{k+1}-1]$, this extension is unique. The uniqueness follows from the facts that every coordinate $(a,c)$ with $c < b$ must be colored light blue and $X$ is southwest deterministic. We denote by $E(a,b,k,\ell)$ this unique extension, if it exists. Now, we wish to prove, by induction on $k$, that the extension $E(a,b,k,\ell)$ always exists, and that all its coordinates outside the rectangle $R(a,b,k,\ell) = [a-2^{k+1}-2\ell, a] \times [b-2^{k+1}, b+2^{k+1}-1]$ are colored with either white, light blue or the border color of the light blue half plane. In the following, we call these colors \emph{uninteresting}.

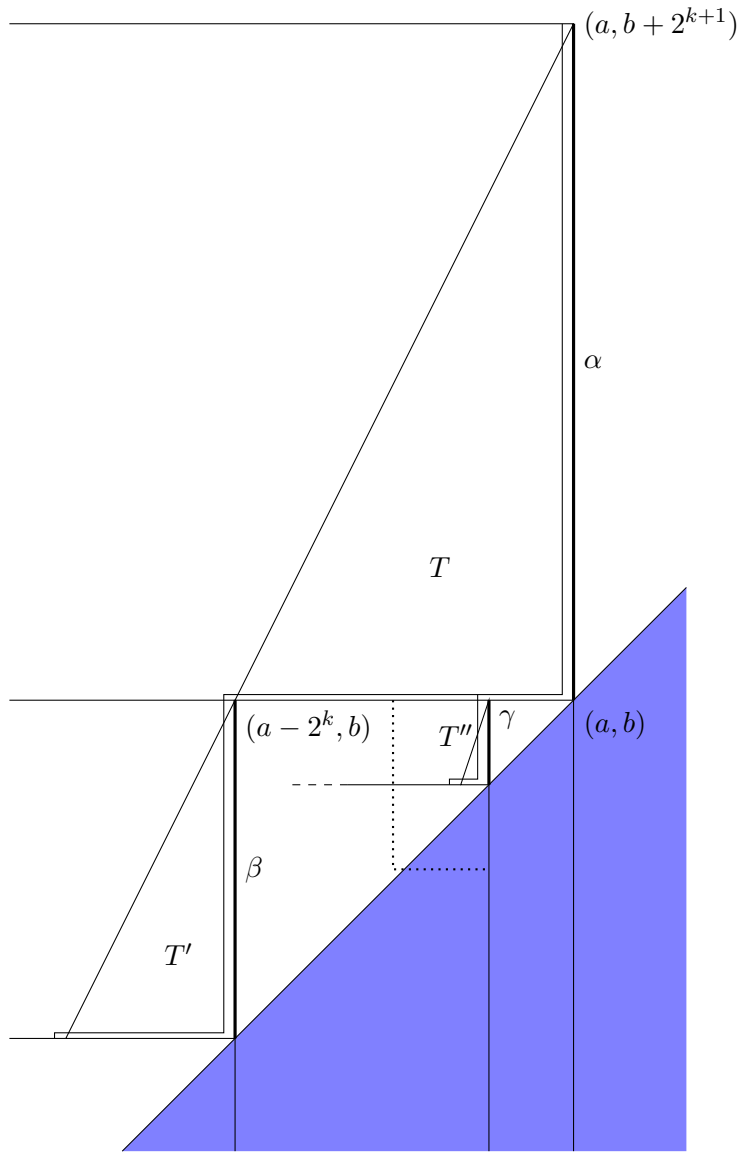
\begin{figure}
\begin{center}
\begin{tikzpicture}[scale = 1.5]

\fill[blue!50] (0,0) -- (5,0) -- (5,5);
\draw (0,0) -- (5,5);

\draw (4,0) -- (4,10) -- (-1,10);
\draw (4,4) -- (1,4) -- (1,1) -- (-.5,1) -- (1,4) -- (4,10);
\draw[very thick] (4,4) -- (4,10);
\draw (-.5,1) -- (-1,1);
\draw (1,0) -- (1,4) -- (-1,4);
\draw[very thick] (1,1) -- (1,4);
\draw (3.25,4) -- (3.25,3.25) -- (3,3.25) -- cycle;
\draw[very thick] (3.25,3.25) -- (3.25,4);
\draw (3.25,0) -- (3.25,4);
\draw (3,3.25) -- (2,3.25);
\draw[dashed] (2,3.25) -- (1.5,3.25);
\draw (3.15,4.05) -- (3.15,3.3) -- (2.9,3.3) -- (2.9,3.25);

\draw (3.9,10) -- (3.9,4.05) -- (.9,4.05) -- (.9,1.05) -- (-.6,1.05) -- (-.6,1);

\draw[thick,dotted] (2.4,4) -- (2.4,2.5) -- (3.25,2.5);

\node[right] () at (4,7) {$\alpha$};
\node[right] () at (1,2.5) {$\beta$};
\node[below right] () at (3.25,4) {$\gamma$};

\node[above left] () at (3,5) {$T$};
\node[] () at (.5,1.75) {$T'$};
\node[above left] () at (3.2,3.5) {$T''$};

\node[below right] () at (4,4) {$(a,b)$};
\node[right] () at (4,10) {$(a,b+2^{k+1})$};
\node[below right] () at (1,4) {$(a-2^k,b)$};

\end{tikzpicture}
\end{center}
\caption{The three red triangles appearing in the induction in the proof of Theorem~\ref{thm:Chain}.}
\label{fig:InductionPic}
\end{figure}

First, consider the case $w < 3 \ell$. Now the determinant of $T$ is $0$ by item~\ref{item:One} above, and the deterministic rule of $X$ creates no other triangles in the quarter-plane, so that the extension clearly exists. The claim about the rectangle $R(a,b,k,\ell)$ also holds, since every color below the horizontal line $y = b$ is uninteresting, and the westmost coordinate of the level stripe of $T$ is $(a-w-2\ell,b+\ell)$, which is inside the rectangle.

In the case $3 \ell \leq w < 6 \ell$, the determinant of $T$ is $1$ by item~\ref{item:Two}. Then, a successor $T'$ of $T$ is created by the deterministic rule assuming that the elimination rule is not triggered. We show that the assumption $3 \ell \leq w$ guarantees it is not. Namely, the branch signal emitted by the point $C$ in Figure~\ref{fig:Chain} moves $2\ell + 2$ steps to the west before turning south. The length of the segment $AC$ is $\frac{3}{4} w \geq \frac{9}{4} \ell \geq 2\ell + 2$, since $\ell \geq 8$ (here we need the restriction that levels below $8$ are not allowed). This means that the elimination rule will not be triggered, as the south branch signal hits the south border of $T$. Now, by the SFT rules, $W(T') = \frac{w}{2} = 2^{k-1}$ and $L(T') = \ell \leq \frac{w}{3} < \frac{5}{3} 2^{k-1}$. The southeast corner of $T'$ is at $(a-w, b-w)$, and the pattern $P(a-w,b-w,k-1,\ell)$ (the area to the left of the line segment $\beta$ in Figure~\ref{fig:InductionPic}) is compatible with $P(a,b,k,\ell)$ in the sense that their union contains no forbidden pattern. The condition of~\eqref{eq:IndHyp} is satisfied by $k-1$ and $\ell$, so by the induction hypothesis, the extension $E(a-w,b-w,k-1,\ell)$ exists. Since the determinant is not $2$, $T$ has no child, and the area below $T$ is filled with the uninteresting colors. Since $R(a-w,b-w,k-1,\ell) \subset R(a,b,k,\ell)$, the claim about the rectange also follows from the induction hypothesis.

If $6 \ell \leq w$, then both a successor $T'$ and a child $T''$ should be created by the deterministic rule. As above, the extension $E(a-w,b-w,k-1,\ell)$ for $T'$ exists and is compatible with $P(a,b,k,\ell)$, so we focus on the hypothetical child $T''$. We have $W(T'') = \frac{w}{8} = 2^{k-3}$ and $L(T'') = \ell+1 \leq \frac{w}{6}+1 < \frac{5}{3} 2^{k-3}$, and the southeast corner of $T''$ is at $(a-2^{k-2},b-2^{k-2})$. Consider the pattern $P(a-2^{k-2},b-2^{k-2},k-3,\ell+1)$, that is, the stripe containing $T''$ (see $\gamma$ in Figure~\ref{fig:InductionPic}), whose domain overlaps with the extension of $T'$. Since $k-3$ and $\ell+1$ satisfy~\eqref{eq:IndHyp}, the extension $E(a-2^{k-2},b-2^{k-2},k-3,\ell+1)$ of $T''$ exists by the induction hypothesis, and all its interesting colors occur in the rectangle $R(a-2^{k-2},b-2^{k-2},k-3,\ell+1)$. Now, this rectangle does \emph{not} overlap with the extension of $T'$, since the $x$-coordinate of its west border is $a - 2^{k-1} - 2 \ell - 2 > a - 2^k$ by the inequality $2 \ell + 2 \leq 2\frac{w}{6} + 2 < 2^{k-1}$. The dotted rectangle in Figure~\ref{fig:InductionPic} represents the rectangle of $T''$. This also implies that the horizontal level stripe of $T''$ and the south border of $T$ do not interact. Thus the pattern obtained by restricting the extension of $T''$ into the rectangle $R(a-2^{k-2},b-2^{k-2},k-3,\ell+1)$ is locally consistent with both the stripe of $T$ and the extension of $T'$. The rest of the quarter plane $Q(a,b,k)$ of $T$ can be consistently filled with the uninteresting colors, and we have obtained the extension $E(a,b,k,\ell)$ of $T$. The claim about the rectangle $R(a,b,k,\ell)$ follows as in the previous case.

Together with an elementary compactness argument, the above induction implies that for all $\ell \geq 8$, there exists a valid configuration $x_\ell \in X$ containing an infinite chain $(T_i)_{i \in \N}$ of red triangles such that each $T_{i+1}$ is the predecessor of $T_i$, the level of each $T_i$ is $\ell$, and the width of each $T_i$ is a power of $2$. To each $T_i$ we associate the maximal chain $(T_{i,j})_{m_i \leq j \leq M_i}$ of red triangles of level $\ell + 1$ such that each $T_{i,j+1}$ is the predecessor of $T_{i,j}$ and $T_{i,M_i}$ is a child of $T_i$. We see that $m_i$ does not depend on $i$ for large enough $i$, and that $M_i$ grows without bound with $i$. This implies that the configuration $x_{\ell + 1} \in X$ is strictly below $x_\ell$ in the subpattern order, since all the initial segments of its respective infinite chain $(T'_i)_{i \in \N}$ can be found in $x_\ell$ as the $(T_{i,j})$. This concludes the proof for the existence of the infinite downward chain in $SP(X)$.

Finally, we prove the countability of $X$. Let $x \in X$ be a valid configuration. The set of configurations containing only infinite red triangles (and thus at most one triangle) in its first layer is seen to be countable by the following case analysis. On the second layer of these configurations, there are no finite segments of the frame. The set of level stripes attached to infinite frame lines can be enumerated simply by the thickness of the signals (which may be infinite), the phase of the decision signal, and the contents of the frame, and thus their set of possibilities is countable.

We then claim that there is at most one maximal dark orange component which is not attached to any frame, and for contradiction we assume there are two, denoted by $O, O' \subset \Z^2$. Now, neither $O$ nor $O'$ has a south border, so for all $\vec n \in O$ ($\vec n \in O'$), we have $\vec n + (2,-1) \in O$ ($O'$, respectively). Thus, there exists a horizontal line that both $O$ and $O'$ intersect, and we can assume that $O$ lies to the west of $O'$. Thus, for all $b \in \Z$ less that some constant $B \in \Z$, we can define $a_b, a'_b \in \Z$ such that $(a_b,b) \in O$, $(a'_b,b) \in O'$ and the distance $a'_b - a_b$ is minimal. In fact, $\delta = a'_b - a_b$ does not depend on $b$ by the shapes of the sets $O$ and $O'$. We may assume that $\delta$ is minimal with respect to all pairs of such dark orange components in $x$. Now, the colors of the line segment between $(a_b,b)$ and $(a'_b,b)$ must be yellow, then dark yellow, and finally orange, in that order (this is forced by the local rules generated by Figure~\ref{fig:Chain}, in particular the elimination rule). Let $c_b, d_b \in \Z$ denote the west and east ends of the dark yellow segment, respectively. Now, it is easy to see that $c_{b-1} = c_b + 2$ and $d_{b+1} = d_b$ hold for all $b \leq B$, so if we denote $e = \left\lceil \frac{1}{2}(d_B - c_B) \right\rceil + 1$, then the segment between $a_{B-e}$ and $a'_{B-e}$ does not contain a dark yellow segment, a contradiction. Thus, there is at most one maximal dark orange component not attached to a frame, and a simple case analysis then shows that the set of all configurations of the second layer without a frame is countable.

On the other hand, suppose $x$ contains at least one finite red triangle in its first layer. Since every finite red triangle must have either a parent or a predecessor, its level is necessarily finite. Let $\ell \geq 8$ be the minimal level of a finite red triangle appearing in $x$, and let $T$ be the smallest red triangle of level $\ell$ appearing in $x$. Without loss of generality we can suppose that the southwest corner of $T$ lies at $(W(T), 2W(T))$. Then, since $T$ determines the positions and levels of its infinite tree of predecessors and their children, the whole of $x|_A$ is determined, where
\[ A = \{ (x,y) \in \Z^2 \;|\; x \geq W(T), x \leq y \leq 2x \}. \]
By the rules of $X_3$, every coordinate to the north of $A$ must be colored white. Consider then the diagonal coordinates $\{ (x,x) \in \Z^2 \;|\; x < W(T) \}$. Attached to this set we may have one red triangle, which must be infinite: if it were finite, the southeast corner of some red triangle in its predecessor/parent chain would be within $A$, a contradiction. The level and position of this infinite triangle, together with the phase of its decision signal and the content of its frame, constitute a countable number of choices for $x$, and after these choices, the second layer is also determined completely.
\end{proof}

We shortly describe the poset structure of $X$ in more detail. At the top of $X$ is the infinite configuration $x_8$ with the smallest possible level, together with those configurations where a single infinite red triangle lies to the west of the inductive structure of some $x_n$. For all $n \geq 8$, we of course have $x_n \succ x_{n+1}$ by the above proof. Directly below $x_n$ is also the configuration $y_n$ with a single infinite red triangle of level $n$. Finally, below all the $x_n$ and $y_n$ we find the configuration with a single infinite red triangle of infinite level, together with the more degenerate ones (containing, for example, nothing but the border of two infinite colored areas).

\subsection{No infinite downward chain with the \PProperty}

We now prove that an SFT with the downward \PProperty cannot be used for this construction, even if countability \emph{or} downward determinism is relaxed. If both the assumption of countability and the assumption of downward determinism are removed, an infinite downward chain is trivially possible, examplified by the SFT
\[ \{x \in \{0,1\}^{\Z^2} \;|\; x = \sigma^{(1,0)}(x)\}. \]

Instead of proving the theorem for the \PProperty directly, we prove the more natural general result that if an SFT is either deterministic with countable projective subdynamics, or is itself countable, then its chains cannot be much longer than the Cantor-Bendixson rank of its projective subdynamics.

\begin{definition}
Let $\lambda$ be an ordinal. An SFT $X$ has the PCB$(\lambda)$ property if its horizontal projective subdynamics is ranked and has Cantor-Bendixson rank at most $\lambda$.
\end{definition}

By Proposition~\ref{prop:SoficFinite} and Lemma~\ref{lem:OrderPreserving}, the \PProperty implies the PCB$(r)$-property for some $r \in \N$. In the next definition, the notation $\sigma^*(x)$ for a configuration $x \in S^\Z$ stands for the \emph{orbit} of $x$, that is, the set $\{ \sigma^n(x) \;|\; n \in \Z \}$.

\begin{definition}
Let $X \subset S^{\Z^2}$ be an SFT, and denote $Y = \PS(X)$. We say $X$ has the \emph{R property} if for all $y \in Y$, there exists $z \in Y^k$ such that $z_1 = y$, $z_k \in \sigma^*(y)$ and $z_i \notin \sigma^*(y)$ for all $i \in [2, k-1]$ with the following property: For any configuration $x \in X$ and indices $m < n \in \Z$ such that $x_n, x_m \in \sigma^*(y)$, but $x_i \notin \sigma^*(y)$ for all $i \in [m+1, n-1]$, we must have $x_{[m, n]} \in \sigma^*(z)$.
\end{definition}

In the above definition, configurations $z$ of the subshift $Y^k$ consist of $k$ configurations of $Y$ stacked on top of each other, and the notation $z_i$ refers to the $i$th of these configurations. Intuitively, $X$ having the R property means that if two rows of $X$ have the same content up to a shift, and no other rows between them do, then this shift and all the rows between them are uniquely determined. By induction, it follows that if a horizontal row repeats up to a shift, the rows in between are taken from a configuration of $X$ with a period $\vec n$ whose y-coordinate is nonzero.

\begin{lemma}
\label{lem:RProperty}
Let $X \subset S^{\Z^2}$ be a countable or downward deterministic SFT defined by Wang tiles. Then, $X$ has the R property.
\end{lemma}

\begin{proof}
To show this, let $y \in Y = \PS(X)$, and let $t \in \Z$ and $x \in X$ be such that $x_i = y = \sigma^{-t}(x_j)$, where $j - i > 0$ is minimal. Define also $z = x_{[i, j]}$. First, suppose $X$ is downward deterministic. Then the claim is obvious: as it is possible for the sequence of rows $z_{[2, j - i + 1]}$ to appear under the row $y$, it is the \emph{only} possible sequence of rows under $y$, and thus repeats infinitely downward, always shifting by $t$ steps when it repeats.

Now, suppose $X$ is countable. We need to show that $z$ is the only word (of rows) in $Y^* = \bigcup_{k \in \N} Y^k$ such that $z_1 = y$, $z_k \in \sigma^*(y)$ and $z_i \notin \sigma^*(y)$ for all $i \in [2, k-1]$. But this is obvious as well, as if we assume the contrary, and let $z'$ be another such word, then $\{z, z'\}^\Z$ is an uncountable set of valid tilings of $X$. 
\end{proof}

\begin{lemma}
\label{lem:ShortChains}
If $X \subset S^{\Z^2}$ is an SFT with the R property with countable projective subdynamics and $x \in X$ has a period, then there do not exist $y, z \in X$ such that $x > y > z$.
\end{lemma}

\begin{proof}
Suppose $x \in X$ has the period vector $\vec n \in \Z^2$. If $\vec n$ has a zero y-coordinate, this means that the rows are periodic, so there are only finitely many different rows in $x$. Some row must then repeat infinitely many times upwards and some row must repeat infinitely many times downward. Since $X$ has the R property, $x$ is in fact horizontally periodic and vertically eventually periodic, and the claim is proved similarly to Lemma~\ref{lem:1DBoringPoset} (although we have not given enough machinery to apply it directly).

Now, assume that $\vec n$ has a nonzero y-coordinate $a$, which we may assume to be positive. As $\PS(X)$ is countable,
\[ Y = \{ z_{[1,a]} \;|\; z \in X, z = \sigma^{\vec n}(z) \} \]
is a countable one-dimensional SFT (since we restrict to the configurations with period $\vec n$). Since $x_{[1,a]} \in Y$, $x$ is horizontally eventually periodic, and the claim follows as above.
\end{proof}

Using the above lemma, we prove an upper bound for the downward chains occurring in an SFT with the R property, in terms of the rank of its projective subdynamics. As a corollary, we obtain the result that infinite downward chains cannot occur in countable SFTs with the \PProperty.

\begin{proposition}
\label{prop:ShortChains}
Let $\lambda$ be an ordinal, and let $X$ be an SFT with the properties R and PCB$(\lambda)$. Then $X$ does not contain a proper downward chain of length $\lambda+2$.
\end{proposition}

\begin{proof}
Assume on the contrary that $(x^\alpha)_{\alpha \leq \lambda+2}$ is such a chain, and consider an arbitrary row $y = x^\alpha_m$, where $\alpha \leq \lambda$ and $m \in \Z$. First, consider the case that $y$ is isolated in the subshift generated by $\rows(x^\beta)$ for some $\beta < \alpha$. Then some word $w \sqsubset y$ isolates $y$ in some $\rows(x^\beta)$ with $\beta < \alpha$, meaning that $w \sqsubset x^\beta$, where $w$ is regarded as a rectangular pattern of height $1$, and every row in $x^\beta$ that contains $w$ is equal to $y$. Because $x^\beta$ is strictly below $x^\alpha$, the pattern $w$ must occur infinitely many times in $x^\beta$, and it then follows from the R property that the rows of $x^\beta$ are eventually periodic up to a horizontal shift. Moreover, the long periodic parts of $x^\beta$ are approximations to $x^\alpha$, so that $x^\alpha$ is periodic with some period vector $\vec n$. By Lemma~\ref{lem:ShortChains}, this is only possible if $\alpha = \lambda + 1$ or $\alpha = \lambda + 2$. Thus, we may restrict to a chain of length $\lambda$ such that for each $y = x^\alpha_m$, $y$ is not isolated in $\rows(x^\beta)$ for any $\beta < \alpha$.

But then, if we let $\lambda_\alpha = \min_{m \in \Z} \rank_{\rows(X)}(x^\alpha_m)$ for all $\alpha \leq \lambda$, it follows from a straightforward transfinite induction that $\lambda_\alpha \geq \alpha$ for all $\alpha \leq \lambda$, and in particular the Cantor-Bendixon rank of $\PS(X)$ is at least $\lambda$. 
\end{proof}

\begin{corollary}
\label{cor:ShortChains}
Let $X$ be a countable or deterministic SFT with the property P. Then for some $k \in \N$, $X$ does not contain a chain of length $k$.
\end{corollary}

\begin{corollary}
Let $X$ be a countable SFT with an infinite downward chain. Then the projective subdynamics of $X$ in any direction have Cantor-Bendixson rank at least $\omega$.
\end{corollary}

A similar result was proved in \cite{BaDuJe08}: for a countable SFT $X$, if the Cantor-Bendixon rank of $X$ is $\lambda$, then there are downward chains of at most length $\lambda$ in $X$.

\section{Embedding finite posets in the subpattern poset}

In \cite{SaTo12b}, we constructed, for each finite poset $P$, a countable two-dimensional SFT $X_P$ whose subpattern poset contains a copy of $P$. We repeat here this construction, with some alterations to ensure determinism and the \PProperty. We also show how to combine the subshifts $X_P$ into a single countable SFT $X$ containing a copy of \emph{every} finite poset. The SFT $X$ is deterministic but by Corollary~\ref{cor:ShortChains} cannot have the \PProperty.

Note that all the embeddings we consider are order-embeddings, that is, any additional relations between elements of the embedded poset are forbidden.

\begin{theorem}
\label{thm:FinitePosets}
Let $(P, \geq)$ be a finite poset. There exists a countable deterministic SFT $X$ with the \PProperty such that $(P, \geq)$ can be order-embedded in $SP(X)$.
\end{theorem}

\begin{proof}
The idea of the construction is the following: Some configurations $x \in X$ correspond to elements of the poset, and such a configuration contains infinitely many special squares. Each square contains a pattern from some other configuration whose poset element is lower than the element corresponding to $x$. The direction of determinism of $X$ is $(1,2)$.

We now present the construction in more detail, constructing inductively the SFT $X$ and a function $f : P \to X$ that gives the desired order-embedding. Let $p \in P$ be arbitrary. We assume that for all $q \in P$ with $q < p$, the configuration $f(q)$ has already been defined, and that there are three colors $0_q, 1_q, 2_q$ such that for all $n \in \N$ larger than a constant $n_q \in \N$, we have $f(q)_{(-n, i)} = 0_q$ and $f(q)_{(n, i)} = 2_q$ for all $i \in \{-n, \ldots, n\}$, and $f(q)_{(i, -n)} = 1_q$ and $f(q)_{(i, n)} \neq 1_q$ for all $i \in \{-n+1, \ldots, n-1\}$. We define three new colors $0_p, 1_p, 2_p$ that will have the same property.

The point $f(p)$ contains an infinite cone extending upwards, with the outside of the cone filled by the three colors. We also have $f(p)_{-i} = \INF 0_p 1_p^{2i+1} 2_p \INF$ for all $i \in \N$, so that the northeast signal $0_p 1_p$ and the northwest signal $1_p 2_p$ meet at the origin, where they produce the base of the cone. The left border of the cone is vertical, and on its right there is a vertical stack of \emph{ruler squares}. The sidelength of the $n$th square is $2(n + N_p) - 1$, where $N_p \in \N$ is the smallest integer such that $n_q \leq 2 N_p - 1$ for all immediate predecessors $q$ of $p$. Note that as we are constructing an SFT and there are finitely many points $f(p)$, we can check that the first ruler square is of the correct size in all of them. The sizes are determined by initializing the width of the first square by a local rule and then incrementing it by $2$ using another local rule, and the squareness is deterministically enforced by a diagonal signal initialized at the southeast corner. See Figure~\ref{fig:Squares} for a visualization.

\input{Squares}

To the right of each ruler square are a number of other squares, called \emph{data squares}, one for each immediate predecessor $q$ of $p$. Note that there may not be any data squares at all in $f(p)$, if $p$ is a bottom element of $P$. The data squares have the same height and width as the ruler square (the widths are enforced by the same method). The data squares are filled with patterns from the $f(q)$. More precisely, the west border of the square is filled with the color $0_q$, the south border with $1_q$, and the east border with $2_q$. Then the central square of $f(q)$ is deterministically formed inside the data square, interrupted by the north border. In the construction, each $f(p)$ will extend the alphabet with completely new symbols (apart from the ones used to simulate the $f(q)$), and each region in the construction will have a different uniform background to differentiate them from each other.

Define $X$ as the orbit closure of $\{f(p) \;|\; p \in P\}$. It is clear that $f$ induces an order-embedding of $P$ into $SP(X)$, and that $X$ is a deterministic SFT with the \PProperty. It remains to be shown that $X$ is countable, and for that, let $x \in X$ be arbitrary. It is an easy case analysis that there are a countable number of configurations that do not contain the base of any cone, so let $p \in P$ be maximal such that $x$ contains the base of the cone of $f(p)$. Now, if $x$ is not a translate of $f(p)$, it must contain a border of the data rectangle of $p$ from $f(q)$, where $q > p$ in $P$. This cannot be the south, west or east border, since otherwise $x$ would contain the entire data square, and then it is easy to see that $x$ in fact contains the base of the cone of $f(q)$, a contradiction with the maximality of $p$. Thus $x$ contain the north border of the data square, for whose position we have countably many choices. Furthermore, the area to the north of that border must be filled with the symbol $1_p$. This proves that $X$ is countable. 
\end{proof}

We can now construct both an infinite downward chain by Theorem~\ref{thm:Chain}, and any finite poset by Theorem~\ref{thm:FinitePosets}. It makes sense to ask if some kind of combination of these two constructions would give us a large class of infinite posets with the ascending chain condition. In fact, this is the case, but the construction is somewhat complicated and will thus be the subject of a future paper. For now, we only prove the following almost-closure property of subpattern posets of countable SFTs.

\begin{theorem}
\label{thm:ManyPosets}
Let $(X_i)_{i \in \N}$ be a computable sequence of countable SFTs, and let $P_i$ be the subpattern poset of $X_i$. There exists a countable SFT $X$ whose subpattern poset is an incomparable disjoint union
\[ P = P' \boxtimes (\boxtimes_{i=0}^\infty) P_i, \]
where $P'$ is a poset of height $4$. If the $X_i$ are northeast deterministic, $X$ can also be made such.
\end{theorem}

Note that the \PProperty is not preserved. If the $X_i$ are deterministic in some directions $\vec{d}_i$ computable from $i$, we can apply transformations by suitable elements of $SL_2(\Z)$ to make them northeast deterministic, since such transformations are effective and preserve the subpattern posets up to an order-isomorphism.

\begin{proof}
We assume the SFTs $X_i$ are defined with allowed patterns of the shape $\{(0,0), (1,0), (1,1)\}$ over the alphabet $\{1, \ldots, n_i\}$. It is well known that every SFT is effectively conjugate to an SFT of this form. Let $M'$ be a deterministic counter machine that, from input $i$, computes $n_i$ and the list of allowed patterns of $X_i$ of this form, and then halts.

We take as a starting point the countable SFT of Example~\ref{ex:Grid} whose configurations are grids of arbitrary size, with the lower right triangle of every grid element colored differently from the top left triangle to ensure square shape. We duplicate the colors, and arrange for the colors used on rows to alternate: Squares of every second row are called \emph{computation squares}, and they form \emph{computation rows}. The other squares are called \emph{signal squares}, and they form \emph{signal rows}. In the top left triangle of each computation square we run a counter machine $M$, starting from the lower left corner of the square. We use the simulation in Construction~\ref{con:CounterMachine}, apart from implementing some operations `directly in the SFT' for synchronization reasons, and from using certain signals of the SFT as `input counters'. We are able to choose the machine $M$ rather independently of the rest of the tiling rules, as it will only affect the rest of the tiling by rejecting or accepting at some moment.

The signal squares are used to send information between computation squares, and a finite (fixed) amount of vertical or diagonal signals are sent through them. The grid and the zig zag heads of the counter machines are illustrated Fig.~\ref{fig:ManyPosets} (a). The counters of the counter machine are not shown. Note that in order to ensure countability, the signals must be represented as borders of differently colored areas. This is not shown in the figures, but the same applies to all signals described below.

Every computation square contains (in addition to the counters of $M$) $4$ named counters: the \emph{index} $i$, the \emph{state} $s$, the \emph{south state} $s_s$ and the \emph{south west state} $s_{sw}$. The index, south state and south west state are immutable input counters, but the counter machine can freely change the value of the state counter. The index counters of squares in the grid are forced to be equal for all machines in any single configuration with a standard signal construction running on the signal rows, see Fig.~\ref{fig:ManyPosets} (b). The counters $s_{sw}$ and $s_s$ are copied from the final contents of the $s$ counter of the computation square in the direction $(-1, -2)$, and the $s$ counter of the computation square to the south, respectively, again with a standard signal construction shown in Fig.~\ref{fig:ManyPosets} (c).

The idea is that the eventual values of the $s$ counters of the machines simulate a scaled version of the SFT $X_i$. The behavior of $M$ is independent of the grid size, except that $M$ must enter an accepting state for the first time \emph{exactly} at the top left corner of the computation squares. Otherwise, a tiling error is produced. The reason for this is to ensure that every SFT $X_i$ can be simulated on a grid of exactly one size so that only one copy of the subpattern poset of $X_i$ appears in the subpattern poset of $X$. The algorithm of $M$ is described in Algorithm~\ref{alg:PosetSim}.

\begin{algorithm}
\caption{The program of the counter machine $M$}\label{alg:PosetSim}
\begin{algorithmic}[1]

\State \textbf{input} $i, s_s, s_{sw} \in \N$
\State $s \gets 0$
\State $n_i, A \gets M'(i)$
\Comment{Colors $[1, n_i]$, allowed triples $A$}
\If{$n_i < s_s \vee n_i < s_{sw}$} \label{ln:Const1}
\Comment{a constant-time operation}
  \State \textbf{reject}
\EndIf
\ForAll{$(a, b, c) \in [1, n_i]^3$} \label{ln:Loop}
\Comment{time $T(i) = \sum_{a, b, c \in \left[1, n_i \right]^3} T'(a, b, c, A)$}
  \If{$(a, b, c) \in A$}
  \Comment{time $T'(a, b, c, A)$}
    \State \Choose{$d \in \{0, 1\}$} \label{ln:DChoice}
    \If{$a = s_{sw} \wedge b = s_s$} \label{ln:Const2}
		\Comment{a constant-time operation}
      \If{$d = 1$}
        \State $s \gets c$ \label{ln:Const3}
				\Comment{$1$ step}
      \Else
        \State \Nop{1}
      \EndIf
    \Else
      \State \Nop{2}
    \EndIf
  \EndIf
\EndFor
\If{$s = 0$}
  \State \textbf{reject}
\Else
  \State \textbf{accept}
\EndIf
\end{algorithmic}
\end{algorithm}

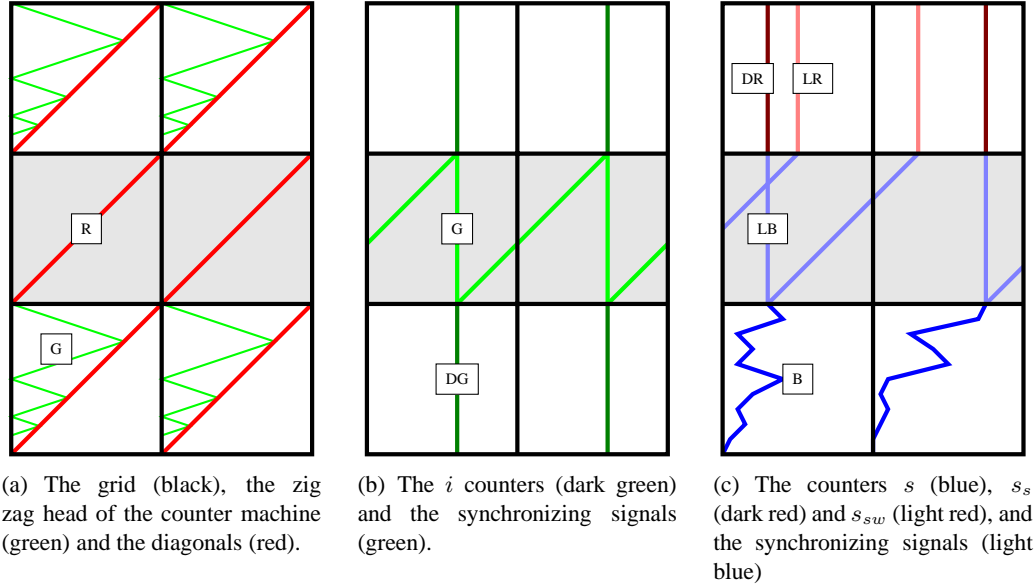
\begin{figure}[!ht]
\centering
\subfigure[The grid (black), the zig zag head of the counter machine (green) and the diagonals (red).]
{
\begin{tikzpicture}[scale = 2]

\colorlet{framecolor}{black}
\colorlet{diagonalcolor}{red}
\colorlet{icolor}{green}
\colorlet{imsgcolor}{green!50!white}
\colorlet{scolor}{blue}
\colorlet{smsgcolor}{blue!50!white}
\colorlet{zigzagcolor}{green}

\coordinate (dsw) at (0,0);
\coordinate (dnw) at (0,1);
\coordinate (ds) at (1,0);
\coordinate (dn) at (1,1);
\coordinate (dse) at (2,0);
\coordinate (dne) at (2,1);

\coordinate (usw) at (0,2);
\coordinate (unw) at (0,3);
\coordinate (us) at (1,2);
\coordinate (un) at (1,3);
\coordinate (use) at (2,2);
\coordinate (une) at (2,3);

\fill[color=black!10!white] (dnw) rectangle (use);

\draw[thick, zigzagcolor] (0,1/8) -- (3/16,3/16) -- (0,1/4) -- (3/8,3/8) -- (0,1/2) -- (3/4,3/4) -- (0,1);
\draw[thick, zigzagcolor] (1,1/8) -- (1+3/16,3/16) -- (1,1/4) -- (1+3/8,3/8) -- (1,1/2) -- (1+3/4,3/4) -- (1,1);
\draw[thick, zigzagcolor] (0,2+1/8) -- (3/16,2+3/16) -- (0,2+1/4) -- (3/8,2+3/8) -- (0,2+1/2) -- (3/4,2+3/4) -- (0,2+1);
\draw[thick, zigzagcolor] (1,2+1/8) -- (1+3/16,2+3/16) -- (1,2+1/4) -- (1+3/8,2+3/8) -- (1,2+1/2) -- (1+3/4,2+3/4) -- (1,2+1);

\draw[ultra thick, diagonalcolor] (dsw) -- (dn) (ds) -- (dne) (dnw) -- (us) (dn) -- (use) (usw) -- (un) (us) -- (une);

\draw[ultra thick, framecolor] (dsw) -- (dse) -- (une) -- (unw) -- cycle;
\draw[ultra thick, framecolor] (dnw) -- (dne) (usw) -- (use) (ds) -- (un);

\picnote{(0.3,0.7)}{G};
\picnote{(0.5,1.5)}{R};

\end{tikzpicture}
}
\quad
\subfigure[The $i$ counters (dark green) and the synchronizing signals (green).]
{
\begin{tikzpicture}[scale = 2]

\colorlet{framecolor}{black}
\colorlet{diagonalcolor}{red}
\colorlet{icolor}{green!50!black}
\colorlet{imsgcolor}{green}
\colorlet{scolor}{blue}
\colorlet{smsgcolor}{blue!50!white}

\coordinate (dsw) at (0,0);
\coordinate (dnw) at (0,1);
\coordinate (ds) at (1,0);
\coordinate (dn) at (1,1);
\coordinate (dse) at (2,0);
\coordinate (dne) at (2,1);

\coordinate (usw) at (0,2);
\coordinate (unw) at (0,3);
\coordinate (us) at (1,2);
\coordinate (un) at (1,3);
\coordinate (use) at (2,2);
\coordinate (une) at (2,3);

\coordinate (dsw-i) at (0.6,0);
\coordinate (dnw-i) at (0.6,1);
\coordinate (usw-i) at (0.6,2);
\coordinate (unw-i) at (0.6,3);
\coordinate (dse-i) at (1.6,0);
\coordinate (dne-i) at (1.6,1);
\coordinate (use-i) at (1.6,2);
\coordinate (une-i) at (1.6,3);

\fill[color=black!10!white] (dnw) rectangle (use);

\draw[ultra thick, icolor] (dsw-i) -- (dnw-i) (dse-i) -- (dne-i) (usw-i) -- (unw-i) (use-i) -- (une-i);
\draw[ultra thick, imsgcolor] (dnw-i) -- (usw-i) (dnw-i) -- (use-i) (dne-i) -- (use-i);

\draw[ultra thick, imsgcolor] (usw-i) -- (0,1.4) (dne-i) -- (2,1.4);

\draw[ultra thick, framecolor] (dsw) -- (dse) -- (une) -- (unw) -- cycle;
\draw[ultra thick, framecolor] (dnw) -- (dne) (usw) -- (use) (ds) -- (un);

\picnote{(0.6,0.5)}{DG};
\picnote{(0.6,1.5)}{G};

\end{tikzpicture}
}
\quad
\subfigure[The counters $s$ (blue), $s_s$ (dark red) and $s_{sw}$ (light red), and the synchronizing signals (light blue)]
{
\begin{tikzpicture}[scale = 2]

\colorlet{framecolor}{black}
\colorlet{diagonalcolor}{red}
\colorlet{scolor}{blue}
\colorlet{smsgcolor}{blue!50!white}
\colorlet{sscolor}{red!50!black}
\colorlet{sswcolor}{red!50!white}

\fill[color=black!10!white] (0,1) rectangle (2,2);

\draw[ultra thick, scolor] (0,0) -- (0.05,0.1) -- (0.15,0.2) -- (0.1,0.3) -- (0.2,0.4) --
              (0.4,0.5)-- (0.1,0.6)-- (0.2,0.7)-- (0.1,0.8)-- (0.4,0.9) -- (0.3,1);
\draw[ultra thick, scolor] (1,0) -- (1,0.1) -- (1.05,0.2) -- (1.1,0.3) -- (1.05,0.4) --
              (1.1,0.5)-- (1.5,0.6)-- (1.4,0.7)-- (1.25,0.8)-- (1.7,0.9) -- (1.75,1);
              
\draw[ultra thick,smsgcolor] (0.3,1) -- (0.3,2) (1.75,1) -- (1.75,2) (0.3,1) -- (1.3,2);
\draw[ultra thick,smsgcolor] (1.75,1) -- (2,1.25) (0,1.5) -- (0.5,2);

\draw[ultra thick,sscolor] (0.3,2) -- (0.3,3) (1.75,2) -- (1.75,3);
\draw[ultra thick,sswcolor] (0.5,2) -- (0.5,3) (1.3,2) -- (1.3,3);

\draw[ultra thick, framecolor] (dsw) -- (dse) -- (une) -- (unw) -- cycle;
\draw[ultra thick, framecolor] (dnw) -- (dne) (usw) -- (use) (ds) -- (un);

\picnote{(0.5,0.5)}{B};
\picnote{(0.2,2.5)}{DR};
\picnote{(0.6,2.5)}{LR};
\picnote{(0.3,1.5)}{LB};

\end{tikzpicture}
}
\caption{Diagrams illustrating the signals of Theorem~\ref{thm:ManyPosets}. In each diagram, six grid cells are shown. The computation rows are white, and the signal rows are shaded.}
\label{fig:ManyPosets}
\end{figure}

We use Construction~\ref{con:CounterMachine} to obtain the SFT rules for the computation squares, apart from lines \ref{ln:Const1}, \ref{ln:Const2} and \ref{ln:Const3}. The comparisons and assignments on these lines need to be implemented so that they take the same amount of time independently of the values in the variables, as we want all computation squares to execute the algorithm in the same amount of time (so that a simulation of a tiling of $X_i$ actually appears on some grid size). This is hard to do with counter machines, but it is very easy to do directly in the SFT by slightly modifying Construction~\ref{con:CounterMachine}: Comparing two variables in one sweep of the zig zag head is a matter of remembering which one is seen first before reaching the right border, and setting the value of counter $i$ to the value of counter $j$ in one sweep can be done by grabbing counter $i$ while moving to the right, and dropping it on the counter $j$ on the way back when the counter $j$ is seen. An extra complication on lines \ref{ln:Const1} and \ref{ln:Const2} is that the input counters $s_s$ and $s_{sw}$ might not yet be visible to the zig zag head, but in this case they are larger than $a$, $b$ and $n_i$, so that comparisons are still doable.

After this implementation detail, it is easy to arrange for the number of steps that $M$ takes on the inputs $i, s_s, s_{sw}$ to be a function of only $i$, as is done on the level of pseudocode in Algorithm~\ref{alg:PosetSim} with sleep commands. Then, for each index $i$, simulations of configurations of $X_i$ happen on grids of exactly one size. A single copy of $P_i$ then appears in the subpattern poset of $X$, and we can distinguish between posets of $X_i$ for different $i$ by looking at the $i$ counter. The extra $P'$ comes from configurations with infinite squares, and the claim about its height is easily checked.

It is easy to check that there are only countably many configurations: For any grid size and choice of $i$, either all configurations of $X_i$ are simulated, or no valid configurations exist. If the $X_i$ are deterministic, we change the line~\algref{alg:PosetSim}{ln:DChoice} to $d \gets 1$. Determinism in the direction $(1, 3)$ (and many others) then follows from the fact that the grid, the signals and the zig zag head of $M$ are all deterministic in this direction. 
\end{proof}

It is easy to see that the construction of Theorem~\ref{thm:FinitePosets} is computable from a given finite poset. We can then apply Theorem~\ref{thm:ManyPosets} to a Turing machine enumerating these SFTs for all finite posets, obtaining the following result.

\begin{corollary}
There exists a deterministic countable SFT $X$ such that the subpattern poset of $X$ contains an embedded copy of every finite poset.
\end{corollary}

\section{Future work}

There is much more to say about the one-dimensional case. It would be interesting to study the derivatives of countable sofic shifts more thoroughly. In particular, it would be interesting to understand which countable SFTs and sofic shifts are `integrable' within SFTs or sofic shifts, and which subshifts have SFT derivatives. For example, the countable SFT $X = \B^{-1}(0^*1^*2^*)$ is not integrable within the class of sofic shifts, since for all $m \in \N$, its integral $Y$ should contain a pattern $a 0^\ell 1^m 2^n b$ with $(a,b) \neq (0,2)$ and arbitrarily large $\ell$ and $n$. But if the integral is sofic, then for some $(a,b) \neq (0,2)$ and $\ell, n \in \N$ there exist infinitely many $m$ as above, and then both $a 0^\ell 1$ and $1 2^n b$ are patterns of $Y^{(1)} = X$, a contradiction (in fact, $X$ is not integrable at all, by a slightly extended argument). An analogous argument shows that $\B^{-1}(0^*10^*20^*)$ is not integrable within the class of sofic shifts either, even though it has the nonsofic integral $\B^{-1}(\{0^k10^\ell10^m20^n \;|\; k, \ell, m, n \in \N, \ell \geq m\})$. Of course, it would also be interesting to know what happens in the case of an uncountable sofic shift as well.

The \PProperty and its relatives are interesting also in the two-dimensional uncountable case: We can construct an uncountable SFT which has the \SoficPProperty (the rows are contained in a countable sofic shift) in all but one direction such that projective subdynamics are countable in every direction. Similarly, it is easy to construct a sofic shift with the \SoficPProperty in every direction which is itself uncountable. However, we do not see how to construct an uncountable SFT which has the \PProperty or \SoficPProperty in every direction. In fact we conjecture that this cannot be done.

\begin{conjecture}
\label{conj:AllPIsCountable}
An uncountable SFT cannot have the \PProperty in every direction.
\end{conjecture}

The construction of a countable SFT with $-\omega$ in the subpattern poset can be generalized to include a much larger class of posets, and this will be the topic of a future paper. On the other hand, it would be interesting to obtain a characterization for the finite posets that are \emph{exactly} realizable as subpattern posets of countable SFTs. In Theorem~\ref{thm:FinitePosets}, we managed to embed every finite poset $P$ into a subpattern poset $\SP(X)$, but $\SP(X)$ contained many more elements than just the copy of $P$. We also know that, for example, the two-element poset $\{0, 1\}$ with $0 < 1$ cannot be realized as the subpattern poset of a countable SFT.

There are some interesting connections between countable SFTs and SFTs with the \PProperty: In \cite{SaTo12c}, we proved that nilpotency is decidable for cellular automata on countable one-dimensional SFTs, because the cellular automaton must either map all configurations to a uniform configuration in a finite number of steps, or it must have a spaceship, and both conditions are easy to detect algorithmically. This implies that it is decidable whether an extendably deterministic SFT with the \PProperty has more than one point. We can prove a similar theorem for SFTs which are either countable or have the \PProperty, using results of \cite{BaDuJe08} in the countable case and techniques of \cite{SaTo12c} in the \PProperty case: Given the forbidden patterns of an SFT $X$ which is either countable or has the \PProperty, it is decidable whether $X$ contains a singly periodic point. However, the problem is known to be undecidable for general SFTs. It would be interesting to investigate such connections further.

\section*{Acknowledgements}

We would like to thank Alexis Ballier, our advisor Jarkko Kari and Charalampos Zinoviadis for fruitful discussions, and the anonymous referees for their suggestions which significantly improved the paper.

\bibliographystyle{fundam}
\bibliography{../../../bib/bib}{}

\end{document}